\documentclass[a4paper,12pt]{amsart}
\usepackage[latin1]{inputenc}

\usepackage{grffile}
\usepackage{longtable}
\usepackage{wrapfig}
\usepackage[normalem]{ulem}
\usepackage{amsmath, amssymb, amsthm}
\usepackage{textcomp,url}

\usepackage[shortlabels]{enumitem} 

\usepackage[colorlinks=true,backref]{hyperref}

\usepackage{pgf,euscript,colortbl,xcolor,times}

\usepackage{a4wide}

\usepackage{graphicx}

\usepackage{mathpazo} 

\numberwithin{equation}{section}

\newcommand{\qand}{\quad\text{and}\quad}

\theoremstyle{plain}
\newtheorem{maintheorem}{Theorem}
\newtheorem{maincorollary}[maintheorem]{Corollary}

\newtheorem{theorem}{Theorem}[section]
\newtheorem{proposition}[theorem]{Proposition}
\newtheorem{corollary}[theorem]{Corollary}
\newtheorem{lemma}[theorem]{Lemma}

\theoremstyle{definition}
\newtheorem{remark}[theorem]{Remark}

\newtheorem{conjecture}{Conjecture}

\newcommand{\RR}{{\mathbb R}}

\newcommand{\ZZ}{{\mathbb Z}}

\newcommand{\EE}{{\mathbb E}}

\newcommand{\DD}{{\mathbb D}}
\newcommand{\sS}{{\mathbb S}}

\newcommand{\ov}{\overline}
\newcommand{\vfi}{\varphi}

\renewcommand{\epsilon}{\varepsilon}
\newcommand{\dist}{\operatorname{dist}}

\newcommand{\closure}{\operatorname{clos}}
\newcommand{\interior}{\operatorname{int}}

\newcommand{\diag}{\operatorname{diag}}

\newcommand{\Leb}{\operatorname{Leb}}
\newcommand{\Lip}{\operatorname{Lip}}
\newcommand{\sing}{\operatorname{Sing}}

\newcommand{\per}{\operatorname{Per}}

\newcommand{\crit}{\operatorname{Crit}}

\newcommand{\supp}{\operatorname{supp}}

\newcommand{\spec}{\operatorname{Spec}}

\newcommand{\cF}{\EuScript{F}}

\newcommand{\cP}{\EuScript{P}}
\newcommand{\cO}{\EuScript{O}}

\newcommand{\V}{\EuScript{V}}
\newcommand{\U}{\EuScript{U}}

\newcommand{\X}{\EuScript{X}}

\newcommand \cD {{\mathcal D}}

\newcommand \cL {{\mathcal L}}

\newcommand{\m}{{\rm Leb}\:}

\author[Vitor Araujo]{Vitor Araujo}
\email{vitor.araujo.im.ufba@gmail.com or vitor.d.araujo@ufba.br}
\urladdr{https://sites.google.com/site/vdaraujo99/}
\address{Instituto de Matemática e Estat\'{\i}stica,
  Universidade Federal da Bahia, Av. Ademar de Barros s/n,
  40170-110 Salvador, Brazil.}



\date{\today}

\title[Number of ergodic physical measures of
  singular-hyperbolic attracting sets]
{On the number of ergodic physical/SRB measures of
  singular-hyperbolic attracting sets}

\thanks{The author was partially supported by
  CNPq-Brazil (grant 300985/2019-3).}

\begin{document}

\begin{abstract}
  It is known that sectional-hyperbolic attracting sets, for a $C^2$
  flow on a finite dimensional compact manifold, have at most finitely
  many ergodic physical invariant probability measures.  We prove an
  upper bound for the number of distinct ergodic physical measures
  supported on a connected singular-hyperbolic attracting set for a
  $3$-flow. This bound depends only on the number of Lorenz-like
  equilibria contained in the attracting set. Examples of
  singular-hyperbolic attracting sets are provided showing that the
  bound is sharp.
\end{abstract}


\maketitle

\tableofcontents

\section{Introduction}
\label{sec:intro}In 1963, the meteorologist Edward Lorenz
published in the Journal of Atmospheric Sciences \cite{Lo63}
an example of a parametrized polynomial system of
differential equations
\begin{align}
  \label{e-Lorenz-system}
\dot X &= a(Y - X) &\quad a &=10
 \nonumber \\
\dot Y &= rX -Y -XZ &\text{where}\qquad\qquad b &=8/3
\\
\dot Z &= XY - bZ &\quad  r &=28
 \nonumber
\end{align}
as a very simplified model for thermal fluid convection,
motivated by an attempt to understand the foundations of
weather forecast.
Numerical simulations performed by Lorenz for an open
neighborhood of the chosen parameters suggested that almost
all points in phase space tend to a \emph{chaotic attractor},
whose well known ``butterfly''  picture can be easily found in the
literature.

The mathematical study of these equations began with the geometric
Lorenz flows, introduced independently by Afraimovich-Bykov-Shil'nikov
\cite{ABS77} and Guckenheimer-Williams \cite{GW79,Wil79} as an
abstraction of the numerically observed features of solutions to
(\ref{e-Lorenz-system}). Tucker \cite{Tucker} showed that the maximal
invariant subset of the classical Lorenz equations
\eqref{e-Lorenz-system} is in fact a geometric Lorenz attractor
through a computer assisted proof.  For more on the rich history of
the study of this system of equations, the reader can consult
\cite{viana2000i, AraPac2010}.

The geometric Lorenz attractor is the most representative example of
the class of singular-hyperbolic flows, an extension of the notion of
(uniform) hyperbolicity encompassing invariant sets with equilibria
accumulated by regular orbits inside the set \cite{MPP99,
  MPP04}. Singular-hyperbolic attracting sets were shown to be
sensitive to initial conditions, to posses finitely many ergodic
physical measures with full basin and with strong statistical
properties; see e.g. \cite{APPV,AMV15,ArMel17,ArMel18}.

Arguably one of the most important concepts in Dynamical Systems
theory is the notion of physical (or $SRB$) measure. We say that an
invariant probability measure $\mu$ for a flow $\phi_t$ induced by a
vector field $G$ is \emph{physical} if the set
\[
B(\mu)=\left\{z\in M:
\lim_{t\to\infty}\frac{1}{t}\int_{0}^{t}\psi(\phi_s(z))\,ds=
\int\psi\, d\mu, \forall \psi\in C^0(M,\RR)\right\}
\]
has non-zero volume, with respect to any volume form on the
ambient compact manifold $M$. The set $B(\mu)$ is by
definition the (ergodic) \emph{basin} of $\mu$. It is assumed that
time averages of these orbits be observable if the flow
models a physical phenomenon.

The study of the existence of these special measures and their
statistical properties for uniformly hyperbolic diffeomorphisms and
flows has a long and rich history, starting with the works of Sinai,
Ruelle and Bowen \cite{Bo75,BR75,Ru76,Ru78,Si72} which were inspired
in the work of Anosov~\cite{An67}.  Hyperbolic attractors (transitive
basic pieces of the Smale ``Spectral Decomposition'') admit a unique
physical measure, as well as transitive Anosov flows, as long as the
smoothness of the underlying vector field is at least $C^2$. For
non-transitive Anosov flows in $3$-manifolds, constructed by Franks
and Williams~\cite{FranksWilliams80}, there are finitely many possible
ergodic physical measures after Brunella~\cite{Brunella93}.

It is well-known that singular-hyperbolic \emph{attractors} (that is,
containing a dense forward regular orbit of the flow) supports a
unique ergodic physical/SRB measure whose basin covers an open
neighborhood the attractor, except for a subset of Lebesgue measure
(volume) zero; see~\cite{APPV}. Analogously singular-hyperbolic
attracting sets (possibly without a transitive trajectory) support
finitely many ergodic physical/SRB measures whose basins together
cover an open neighborhood the attractor, again with the exception of
a zero volume subset; see e.g. \cite{araujo_2021} and references
therein.

In this work, we study the number of distinct ergodic physical
measures supported on a singular-hyperbolic attracting sets and
provide an upper bound for this number, depending only on the number
of Lorenz-like equilibria contained in the attracting set. Examples of
singular-hyperbolic attracting sets are provided showing that the
bound is sharp.

\subsection{Definitions and statements of results}
\label{sec:statements}

Let $M$ be a compact connected Riemannian manifold with dimension
$\dim M=m$, induced distance $d$ and volume form $\m$. Let $\X^r(M)$,
$r\ge1$, be the set of $C^r$ vector fields on $M$ and denote by
$\phi_t$ the flow generated by $G\in\X^r(M)$. For any given subset
$S\subset M$ we denote by $\closure{S}$ the topological closure of
$S$.

\subsubsection{Sectional-hyperbolic attracting sets}
\label{sec:PH}

An \emph{invariant set} $\Lambda$ for the flow $\phi_t$ is a
subset of $M$ which satisfies $\phi_t(\Lambda)=\Lambda$ for
all $t\in\RR$.  Given a compact invariant set $\Lambda$ for
$G\in \X^r(M)$, we say that $\Lambda$ is \emph{isolated} if
there exists an open set $U\supset \Lambda$ such that
$ \Lambda =\bigcap_{t\in\RR}\closure{\phi_t(U)}$.  If $U$ can
be chosen so that $\closure{\phi_t(U)}\subset U$ for all
$t>0$, then we say that $\Lambda$ is an \emph{attracting
  set} and $U$ a \emph{trapping region} (or \emph{isolated
  neighborhood}) for
$\Lambda=\Lambda_G(U)=\cap_{t>0}\closure{\phi_t(U)}$.

For a compact invariant set $\Lambda$, we say that $\Lambda$ is {\em
  partially hyperbolic} if the tangent bundle over $\Lambda$ can be
written as a continuous $D\phi_t$-invariant sum
$ T_\Lambda M=E^s\oplus E^{cu}, $ where $d_s=\dim E^s_x\ge1$ and
$d_{cu}=\dim E^{cu}_x\ge2$ for $x\in\Lambda$, and there exist
constants $C>0$, $\lambda\in(0,1)$ such that for all $x \in \Lambda$,
$t\ge0$, we have
\begin{itemize}
\item \emph{uniform contraction along} $E^s$:
  $ \|D\phi_t | E^s_x\| \le C \lambda^t; $ and
\item \emph{domination of the splitting}:
  $ \|D\phi_t | E^s_x\| \cdot \|D\phi_{-t} | E^{cu}_{\phi_tx}\|
  \le C \lambda^t.  $
\end{itemize}
We say that $E^s$ is the \emph{stable bundle} and $E^{cu}$
the \emph{center-unstable bundle}.  A {\em partially
  hyperbolic attracting set} is a partially hyperbolic set
that is also an attracting set.

We say that the center-unstable bundle $E^{cu}$ is
\emph{volume expanding} if there exists $K,\theta>0$ such
that $|\det(D\phi_t| E^{cu}_x)|\geq K e^{\theta t}$ for all
$x\in \Lambda$, $t\geq 0$. More generally, $E^{cu}$ is {\em
  sectional expanding} if for every two-dimensional subspace
$P_x\subset E^{cu}_x$,
\begin{align} \label{eq:sectional} |\det(D\phi_t(x)\mid P_x
  )| \ge K e^{\theta t}\quad\text{for all $x \in \Lambda$,
    $t\ge0$}.
\end{align}

If $\sigma\in M$ and $G(\sigma)=0$, then $\sigma$ is called an {\em
  equilibrium} or \emph{singularity} in what follows and we denote by
$\sing(G)$ the family of all such points. We say that a singularity
$\sigma\in\sing(G)$ is \emph{hyperbolic} if all the eigenvalues of
$DG(\sigma)$ have non-zero real part.

A point $p\in M$ is \emph{periodic} for the flow $\phi_t$ generated by
$G$ if $G(p)\neq\vec0$ and there exists $\tau>0$ so that
$\phi_\tau(p)=p$; its orbit
$\cO_G(p)=\phi_{\RR}(p)=\phi_{[0,\tau]}(p)$ is a \emph{periodic
  orbit}, an invariant simple closed curve for the flow. The family of
periodic orbits of $G$ is written $\per(G)$.

The \emph{critical elements} $\crit(G)$ of a vector field $G$ are its
equilibria and periodic orbits, that is,
$\crit(G)=\sing(G)\cup\per(G)$.  An invariant set is \emph{nontrivial}
if it is not a critical element of the vector field.

We say that a compact invariant set $\Lambda$ is a
\emph{sectional hyperbolic set} if $\Lambda$ is partially
hyperbolic with sectional expanding center-unstable bundle
and all equilibria in $\Lambda$ are hyperbolic.  A sectional
hyperbolic set which is also an attracting set is called a
{\em sectional hyperbolic attracting set}.

A \emph{singular hyperbolic set} is a compact invariant set
$\Lambda$ which is partially hyperbolic with volume
expanding center-unstable subbundle and all equilibria
within the set are hyperbolic. A sectional hyperbolic set is
singular hyperbolic and both notions coincide if, and only
if, $d_{cu}=2$. In what follows, singular-hyperbolicity implicity
means that $d_{cu}=2$.

\begin{lemma}[Hyperbolic Lemma]{\cite[Lemma 3]{MPP99}
  (or~\cite[Proposition 6.2]{AraPac2010})}
  \label{le:hyplemma}
  A sectional hyperbolic set with no equilibria is a \emph{(uniformly)
    hyperbolic set}.
\end{lemma}

This lemma precisely means that, in our setting, the central unstable
subbundle admits a splitting $E^{cu}_x=\RR\{G(x)\}\oplus E^u_x$ for
all $x\in\Lambda$ (so that $d_{cu}=1+d_u$ with $d_u=\dim E^u$) where
$E^u_x$ is uniformly contracting under the time reversed flow; see
e.g.~\cite{AraPac2010}. That is, $\Lambda$ is a (uniformly)
\emph{hyperbolic set if} $T_\Lambda M=E^s\oplus\RR\{G\}\oplus E^u$,
where $E^s$ is uniformly contracting, $\RR\{G\}$ is the
one-dimensional subspace along the direction of the vector field $G$,
and we have for the same $C,\lambda$ used in the uniform contraction
of $E^s$
\begin{itemize}
\item \emph{uniform expansion (backward contraction) along} $E^u$:
  $ \|D\phi_{-t} | E^u_x\| \le C \lambda^t$ forall $x\in\Lambda$ and
   $t\ge0$.
\end{itemize}

A \emph{periodic orbit $\cO_G(p)$ is hyperbolic} if $\cO_G(p)$ is a
hyperbolic subset for $G$. If moreover $E^u$ is trivial (i.e.
$E^u_q=\{\vec0\}, q\in\cO_G(p)$), then the periodic orbit is a
\emph{periodic sink}.

\begin{remark}\label{rmk:noisol}
  A singular hyperbolic attracting set cannot contain isolated
  periodic orbits.  For otherwise such orbit must be a periodic sink,
  contradicting volume expansion.
\end{remark}

We recall that a subset $\Lambda \subset M$ is \emph{transitive} if it
has a full dense orbit, that is, there exists $x\in \Lambda$ such that
$\closure{\{\phi_tx:t\ge0\}}=\Lambda= \closure{\{\phi_tx:t\le0\}}$.

A nontrivial transitive sectional hyperbolic attracting set is a
\emph{sectional hyperbolic attractor}.

The prototype of a sectional-hyperbolic attractor for $3$-flows is the
Lorenz attractor; see e.g. \cite{Lo63,Tu99,AraPac2010}. For higher
dimensional flows we have the multidimensional Lorenz attractor; see
\cite{BPV97}. More examples are provided in
Section~\ref{sec:geompropert} and many more in~\cite{Morales07}.

\subsubsection{Invariant manifolds}
\label{sec:invari-manifolds-non}

An embedded disk $\gamma\subset M$ is a (local) {\em
  strong-unstable manifold}, or a {\em strong-unstable
  disk}, if $\dist(\phi_{-t}(x),\phi_{-t}(y))$ tends to zero
exponentially fast as $t\to+\infty$, for every
$x,y\in\gamma$. In the same way $\gamma$ is called a (local)
{\em strong-stable manifold}, or a {\em strong-stable disk},
if $\dist(\phi_{t}(x),\phi_{t}(y))\to0$ exponentially fast as
$n\to+\infty$, for every $x,y\in\gamma$.

It is well-known that there exists $\epsilon_0>0$ so that every point
in a hyperbolic set possesses a local strong-stable manifold
$W_{loc}^{ss}(x)$ and a local strong-unstable manifold
$W_{loc}^{uu}(x)$ which are disks tangent to $E^s_x$ and $E^u_x$ at
$x$ respectively with topological dimensions $d_s=\dim(E^s)$ and
$d_u=\dim(E^u)$ and inner radius $\epsilon_0$; see e.g. \cite[Chap.
6]{fisherHasselblatt12}. It is common to write
$W^*_\epsilon(x), *=ss,uu$ for the corresponding local manifolds with
inner radius $\epsilon$.

Considering the action of the flow we get the (global)
\emph{strong-stable manifold}
$$W^{ss}(x)=\bigcup_{t>0}
\phi_{-t}\Big(W^{ss}_{loc}\big(\phi_t(x)\big)\Big)$$ and the (global)
\emph{strong-unstable manifold}
$$W^{uu}(x)=\bigcup_{t>0}\phi_{t}\Big(W^{uu}_{loc}\big(\phi_{-t}(x)\big)\Big)$$
for every point $x$ of a uniformly hyperbolic set. Similar notions are
defined in a straightforward way for diffeomorphisms. These are
immersed submanidfolds with the same differentiability of the flow or
the diffeomorphism.

In the case of a flow we also consider the \emph{stable
  manifold} $W^s(x)=\cup_{t\in\RR} \phi_{t}\big(W^{ss}(x)\big)$
and \emph{unstable manifold}
$W^u(x)=\cup_{t\in\RR}\phi_{t}\big(W^{uu}(x)\big)$ for $x$ in a
uniformly hyperbolic set, which are flow invariant.

We note that these notions are well defined for a hyperbolic
periodic orbit, since this compact set is itself a
hyperbolic set. Since all periodic orbits in a singular-hyperbolic set
are hyperbolic, then these manifolds also exist in this setting.

In general, (local) stable manifolds exist for every point of a
singular-hyperbolic set due to partial hyperbolicity; see
Subsection~\ref{sec:existence-stable-fol} and~\cite{ArMel17}.

\subsubsection{Singularities in singular-hyperbolic attracting sets}
\label{sec:singul-singul-hyperb}

\begin{proposition}{\cite[Proposition
    2.1]{araujo_2021}} \label{prop:generaLorenzlike} Let $\Lambda$ be
  a sectional hyperbolic attracting set and let $\sigma\in\Lambda$ be
  an equilibrium.  If there exists $x\in\Lambda\setminus\{\sigma\}$ so
  that $\sigma\in\omega(x)\cup\alpha(x)$, then $\sigma$ is
  \emph{generalized Lorenz-like}: that is, $DG(\sigma)|E^{cu}_\sigma$
  has a real eigenvalue $\lambda^s$ and
  $\lambda^u=\inf\{\Re(\lambda):\lambda\in\spec(DG(\sigma)),
  \Re(\lambda)\ge0\}$ satisfies $-\lambda^u<\lambda^s<0<\lambda^u$ and
  so the index of $\sigma$ is $\dim E^s_\sigma=d_s+1$.
\end{proposition}

\begin{remark}
  \label{rmk:notLorenzlike}
  \begin{enumerate}
  \item Partially hyperbolicity of $\Lambda$ ensures that the direction
    $G(x)$ of the flow is contained in the central-unstable subbundle
    $G(x)\subset E^{cu}_x$ for all $x\in\Lambda$; see \cite[Lemma
    5.1]{ArArbSal}
  \item If $\sigma\in\sing(G)\cap\Lambda$ is a generalized
    Lorenz-like singularity and $\gamma_\sigma^s$ is its
    local stable manifold, then at
    $w\in\gamma_\sigma^s\setminus\{\sigma\}$ we have
    $T_w\gamma_\sigma^s=
    E^{cs}_w=E^s_w\oplus\RR\cdot\{G(w)\}$ since
    $T\gamma_\sigma^s$ is $D\phi_t$-invariant and contains
    $G(w)$ (because $\gamma_\sigma^s$ is $\phi_t$-invariant)
    and the dimensions coincide.
  \item If $\sigma\in\sing(G)\cap\Lambda$ is a generalized
    Lorenz-like singularity, then the strong-stable manifold of
    $\sigma$ (with dimension $d_s=\dim E^s$), that is
    \begin{align*}
      W^{ss}_\sigma=\left\{x\in M: \dist(\phi_t(x),\sigma) e^{-\lambda^s
      t} \xrightarrow[t\to+\infty]{}0 \right\} 
    \end{align*}
    does not intersect any other point of $\Lambda$:
    $W^{ss}_\sigma\cap\Lambda=\{\sigma\}$; see e.g.\cite[Lemma 5.30 \&
    Remark 5.31]{AraPac2010}.
  \item If an equilibrium $\sigma\in\sing(G)\cap\Lambda$ is not
    generalized Lorenz-like, then $\sigma$ is not in the limit set of
    $\Lambda\setminus\{\sigma\}$, i.e. there is no
    $x\in\Lambda\setminus\{\sigma\}$ so that
    $\sigma\in\alpha(x)\cup\omega(x)$. An example is provided by the
    pair of equilibria of the Lorenz system of equations away from the
    origin: these are saddles with an expanding complex eigenvalue
    which belong to the attracting set of the trapping ellipsoid
    already known to E. Lorenz; see e.g. \cite[Section
    3.3]{AraPac2010} and references therein and also
    Subsections~\ref{sec:lorenz-attract-set}
    and~\ref{sec:examples-with-sharp}.
  \end{enumerate}
\end{remark}

\subsubsection{Statements of the results}
\label{sec:statement-results}

Singular-hyperbolic (and sectional-hyperbolic) \emph{attracting sets}
for $C^2$ smooth flows admit \emph{finitely many ergodic physical
  measures}; see e.g. \cite{ArSzTr} and \cite{araujo_2021}. Here we
provide a bound for the number of such ergodic physical measures. One
of the motivations for our statement comes from the following result
of Morales~\cite{morales04}.

\begin{theorem}\label{thm:morales}
  Let $\Lambda$ be a singular-hyperbolic attractor of a $3$-flow of a
  $C^r$ vector field $X$, $r \ge 1$. Then, there is a neighborhood $U$
  of $\Lambda$ such that \emph{every attractor} in $U$ of a $C^r$
  vector field $C^r$ close to $X$ is singular.
\end{theorem}

Therefore, in this setting, it is natural to consider the number of
ergodic physical measures whose support contains a singularity, since
there are no hyperbolic attractors. In addition, from
Proposition~\ref{prop:generaLorenzlike}, \emph{any singularty
  contained in the support of these physical measures are necessarily
  (generalized) Lorenz-like}, since the support of an ergodic measure
for a continuous invertible map on a metric space admits a dense
forward and backward orbit; see e.g. \cite{Man87}.

We recall that all hyperbolic sets for flows are singular-hyperbolic
in particular. Moreover, for $C^2$ smooth flows, each hyperbolic
attractor admits a unique physical measure (see e.g. Bowen-Ruelle
\cite{BR75} or \cite[Theorem 7.4.10]{fisherHasselblatt12}), then we
can find flows with any number of physical measures within hyperbolic
attracting sets; see below and Subsection~\ref{sec:finiteHypatt}.

\begin{maintheorem}\label{mthm:numberSRB}
  Let $G$ be a $3$-vector field of class $C^2$, $\Lambda$ be a connected
  singular-hyperbolic attracting set of $G$, and $s_L$ be the number
  of Lorenz-like singularities of $\Lambda$. Then the number $s$ of
  ergodic physical measures supported in $\Lambda$ whose support
  contains a singularity satisfies $s\le2\cdot s_L$.
\end{maintheorem}

Let $\V$ be the $C^r$ neighborhood of a $3$-vector field $X$ which
admits a singular-hyperbolic attractor $\Lambda$ with trapping region
$U$, for some $r\ge2$, according to the previous
Theorem~\ref{thm:morales}. Then the previous bound applies to all
physical measures of the attracting set within $U$ for all vector
fields in $\V$. More precisely, we have the following.
  
  \begin{maincorollary}\label{mcor:numberSRB}
    Let $G\in\V$ be given.  Let $s_0$ be the number of singularities
    of $G$ in $U$. Then the number $s$ of ergodic physical measures
    supported in $\Lambda=\Lambda_G(U)$ satisfies $s\le2\cdot s_0$
    (and all of them contain some singularity in their support).
  \end{maincorollary}

  About the inequality above, we observe the following.

  \begin{enumerate}
  \item If $\Lambda$ contains no equilibria of $G$, that is, $s_L=0$,
    then we have a hyperbolic attracting set, so we obtain equality in
    the bound given by Theorem~\ref{mthm:numberSRB} in this particular
    case: $s=0$ when $s_L=0$.

    In Subsection~\ref{sec:finiteHypatt}, we provide a construction of
    a connected singular-hyperbolic attracting set with no
    singularities (and thus, uniformly hyperbolic) and any given
    finite number of ergodic physical measures (none of which contain
    equilibria).
    
  \item Moreover, Morales~\cite[Theorem A]{Morales07} describes the
    construction of a singular-hyperbolic attracting set whose unique
    equilibrium is non-Lorenz-like, and admits a non-singular
    transitive component which supports an ergodic physical
    measure. We again have $s=0=s_L$ but with a non-Lorenz-like
    equilibrium.

  \item In addition, the (geometric) Lorenz attractor, provides an
    example where $s_L=s=1$; see
    Subsection~\ref{sec:lorenz-attract-set}. Hence, a strict
    inequality is obtained in the statement of
    Theorem~\ref{mthm:numberSRB}.

  \item It is easy to increase the number of Lorenz-like and
    non-Lorenz like equilibria while keeping the number of ergodic
    physical measures: see
    Remarks~\ref{rmk:moreLorenzlike},~\ref{rmk:morenonL}
    and~\ref{rmk:nonLlike}. Moreover, Morales~\cite[Theorem
    B]{Morales07} presents an example of a singular-hyperbolic
    attractor with several Lorenz-like equilibria.
  \end{enumerate}

  We also describe examples showing that the inequality is sharp.
  
  \begin{maintheorem}\label{mthm:sharpex}
    Given a integer $s_L>0$ there exists a $C^2$ vector field, on a
    bounded and open subset of $\RR^3$, having a connected
    singular-hyperbolic attracting set $\Lambda$ supporting exactly
    $2s_L$ ergodic physical measures and $s_L$ equilibria, all of them
    Lorenz-like: i.e. $s=2\cdot s_L$.
  \end{maintheorem}

  The constructions can be easily adapted to provide examples where
  the number of (hyperbolic non-Lorenz-like) equilibria in the
  singular-hyperbolic attracting set is larger than the number of
  Lorenz-like equilibria; see Remark~\ref{rmk:notLorenzlike} after the
  proof of Theorem~\ref{mthm:sharpex}.


\subsection{Possible extension of the results}
\label{sec:organization}

The above results should be extended, either for higher dimensional
manifolds, or for less smooth vector fields. Moreover, the examples
presented in the proof of Theorem~\ref{mthm:sharpex} are not robust,
so the inequality $s<s_L$ should be generic, and the variation of the
number of physical measures with the vector field should be considered.

\subsubsection{Semicontinuity of the number of physical measures}
\label{sec:semicont-number-phys}

Since we have obtained an upper bound for the number of singular
physical measures (the ones containing some equilibria in their
support) on singular-hyperbolic attracting sets; these attracting
sets are robust (for each $C^1$ close vector field the trapping region
still contains a singular-hyperbolic attracting set); and the number
of Lorenz-like equilibria is locally constant in a $C^1$ neighborhood
of the vector field, it is natural to conjecture the following.

  Moreover, the number of ergodic physical measures containing
  singularties in their support varies upper semicontinuously with the
  vector field, but a pair of physical ergodic measures may fuse under
  arbitrarly small perturbations.

  \begin{conjecture}\label{conj:semicont}
    Let $\V$ be the open family of $C^2$ vector fields admitting a
    singular-hyperbolic attracting set $\Lambda$ and $s:\V\to\ZZ_0^+$
    the function associating to each $G\in\V$ the number $s(G)$ of
    ergodic physical measures supported in $\Lambda$ and containing
    some equilibria. Then $s$ is upper semicontinous.
  \end{conjecture}

\subsubsection{Genericity of the inequality $s<s_L$}
\label{sec:generic-inequal-ss_l}

The existence of periodic orbits at the boundary of the transversal
section in the examples presented in the proof of
Theorem~\ref{mthm:sharpex} (see Section~\ref{sec:geompropert}) is not
a robust situation under small perturbations of the vector
field. This, it is natural to state the following.

\begin{conjecture}\label{conj:ineqgeneric}
  Among the family of $3$-vector fields $G$ of class $C^2$ exhibiting
  singular-hyperbolic attracting set containing some Lorenz-like
  singularity, there exists an open and dense subset of vector fields
  for which the bound $s<s_L$ holds, for each singular-hyperbolic
  attracting set.
\end{conjecture}

\subsubsection{Extension to higher dimensional vector fields}
\label{sec:extens-higher-dimens}

The results stated above hold for $3$-vector fields of class $C^2$
having singular-hyperbolic attracting sets. But the existence of
finitely many ergodic physical measures also holds for higher
dimensional singular-hyperbolic attracting sets with any stable
dimension, that is, $d_s>1$ and $d_{cu}=2$. Recently this was extended
in \cite{araujo_2021} to sectional-hyperbolic attracting sets with any
combination of $d_s\ge1$ and $d_{cu}\ge2$ values. Moreover, the
existence of homoclinic classes in sectional-hyperbolic attracting
sets was obtained in\cite{ArbBarMor16}. So it is natural to state the
following.

\begin{conjecture}\label{conj:sectionalhyp}
  The results of Theorem~\ref{mthm:numberSRB} and
  Corollary~\ref{mcor:numberSRB} also hold for any
  sectional-hyperbolic attracting sets, replacing ``Lorenz-like''
  singularities by ``generalized Lorenz-like'' singularities.
\end{conjecture}

Since part of the statement of the results is inspired on
Theorem~\ref{thm:morales} it is only natural to propose the following.

\begin{conjecture}\label{conj:nohipattract}
  The statement of Theorem~\ref{thm:morales} holds also for
  sectional-hyperbolic attracting sets $C^1$ close to
  sectional-hyperbolic attractors, without dimensional restrictions
  (i.e. admitting any combination of $d_s\ge1$ and $d_{cu}\ge2$).
\end{conjecture}

\subsubsection{Extension to $C^1$ smooth vector field}
\label{sec:extension-c1-smooth}

Recently existence and uniqueness of physical/SRB measures was obtained
for any uniformly hyperbolic attractor of a $C^1$ generic
diffeomorphism on any compact manifold by Qiu~\cite{qiu11}.

More recently, Crovisier-Yang-Zhang~\cite[Corollary B.1 \& Theorem
J]{CYZ20} obtained existence of a SRB invariant probability measure
for any sectional-hyperbolic attracting set for a $C^1$ vector field
and, assuming that this SRB measure is unique, they are able to deduce
that it is also the unique physical measure.  This points to the
following possibility.

\begin{conjecture}\label{conj:C1SRB}
  For generic $C^1$ vector fields having a sectional-hyperbolic
  attracting set the number of ergodic physical/SRB measures satisfies
  the same bounds as given by Theorem~\ref{mthm:numberSRB} and
  Corollary~\ref{mcor:numberSRB}.
\end{conjecture}

For more on SRB versus physical measures, see
Section~\ref{sec:suppreghyp}.

\subsection{Organization of the text}
\label{sec:organization-text}

In the following, Section~\ref{sec:geompropert} contains descriptions
of the construction of the examples mentioned after the statement of
Corollary~\ref{mcor:numberSRB} including Theorem~\ref{mthm:sharpex}.

In Section~\ref{sec:singhypexist}, we present properties of
singular-hyperbolic attracting sets for $3$-flows in
Subsections~\ref{sec:suppreghyp} and~\ref{sec:coroll-non-existence},
which will be necessary for the proof of the bounds in
Theorem~\ref{mthm:numberSRB} and Corollary~\ref{mcor:numberSRB} in
Subsection~\ref{sec:number-singul-physic}.  

\subsection*{Acknowledgments}

We thank the Mathematics and Statistics Institute of the Federal
University of Bahia (Brazil) for its support of basic research and
CNPq (Brazil) for partial financial support.


\section{Description of classes of examples}
\label{sec:geompropert}

Here we describe the construction of the examples mentioned in
Theorem~\ref{mthm:sharpex}.

\subsection{Connected attracting set containing any finite number
  of hyperbolic attractors}
\label{sec:finiteHypatt}

\subsubsection{Criteria for connectedness of attracting sets}
\label{sec:criter-connect-attra}

We start with a simple criteria to obtain a connected attracting set.

\begin{lemma}\label{le:connectedLambda}
  If $U$ is a connected open subset of $M$ which is a trapping region,
  then the attracting set $\Lambda=\bigcap_{t>0}\closure{\phi_t(U)}$
  is a compact and connected subset.
\end{lemma}

\begin{proof} Indeed, the assumption on $U$ ensures that
  $\phi_t(U)$ is connected and its closure is contained in $U$ for
  all $t\ge T$.  Since the attracting set can be written
\[\Lambda=\bigcap_{t>0}\closure{\phi_t(U)}=\lim_{t\to+\infty}\closure{\phi_t(U)}\]
where the limit is in the Hausdorff topology of compact subsets of
Euclidean $3$-space, then $\Lambda$ is compact and connected.  More
precisely, for each $\epsilon>0$ there exists $T_0>0$ so that for all
$T>T_0$ we get
\[
  \Lambda
  \subset
  \bigcap_{0<t\le T}\closure{\phi_t(B)} =
  \closure{\phi_{T}(B)}
  \subset B(\Lambda, \epsilon)
  =\bigcup_{x\in\Lambda} B(x,\epsilon).
\]
Arguing by contradiction, if we assume that $\Lambda$ is not
connected, then there would exist two disjoint non-empty
closed subsets $\Lambda_1\cup\Lambda_2=\Lambda$. Hence there
would also exist disjoint open neighborhoods $W_i$ of
$\Lambda_i, i=1,2$ with $W=W_1\cup W_2$ an open neighborhood
of $\Lambda$. Thus we can find $T_0>0$ so that
$\closure{\phi_t(U)}\subset W$ for all $t>T_0$ and, by connectedness
of $\phi_t(U)$, this subset is contained in either $W_1$ or
$W_2$. Moreover, by continuity of the flow, $\phi_t(U)$ is
contained in the same $W_i$, let us say $W_1$.

But then $\Lambda=\bigcap_{t>0}\closure{\phi_t(U)}\subset \Lambda_1$
contradicting the existence of nonempty disjoint closed
subsets $\Lambda_1,\Lambda_2$ whose union is $\Lambda$.
\end{proof}

\subsubsection{A Morse-Smale of the disk with an
  attracting set containing finitely many attractors}
\label{sec:hyperb-diffeom-disk}

We consider the plane flow $\psi_t$ defined by the ODE
$\dot x= g'(x) , \dot y=-y$ with $g(x)=\sin(\pi x)$ on the rectangle
$R_k=[-1,2k]\times[-1,1]$ for some fixed integer $k>0$.

Then $R_k$ contains $k+1$ sinks at
$S_0=\{2i-1/2: i=0,\dots,k\}\times\{0\}$ and $k$ saddles
$S_1=\{2i+1/2: i=0,\dots,k-1\}\times\{0\}$ and
$\psi_t(R_k)\subset\interior{R_k}$ for all $t>0$; see the left hand
side of Figure~\ref{fig:MS}.

\begin{figure}[htb]
\centering
\includegraphics[width=8cm]{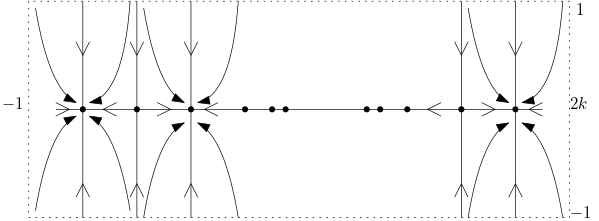}
\includegraphics[width=6cm]{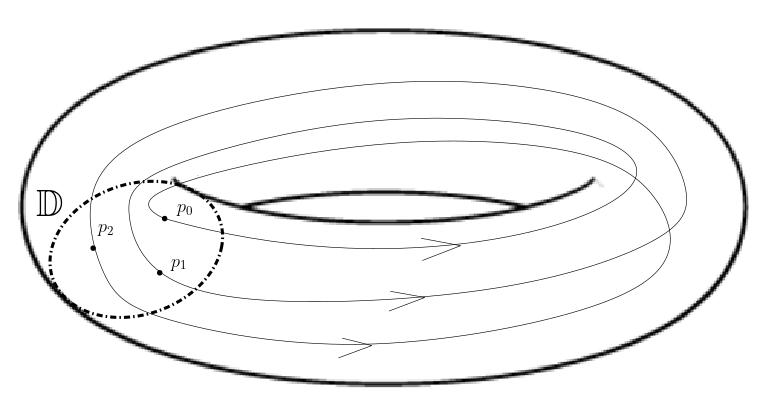}
\caption{\label{fig:MS} The phase portrait of the plane flow $\psi_t$,
  on the left. On the right: suspension of a Plykin attractor for the
  Poincar\'e first return map to the disk $\DD$ of a flow inside a
  solid torus $\DD\times\sS^1$.}
\end{figure}

We note that $R_k$ is diffeomorphic to the disk and the maximal
invariant set $\Lambda=\bigcap_{t>0}\phi_t(R_k)$ given by
$[0,2k-1]\times\{0\}$ is an attracting set.

We also consider the vector field $X$ in the solid torus
$R_k\times\sS^1$ given by $X(x,y,z)=(g'(x),-y,1)$ in what follows.

\subsubsection{The suspension of the Plykin attractor on the disk}
\label{sec:plykin-attractor}

For the construction of the Plykin attractor see Plykin \cite{Pl74},
Robinson\cite{robinson1999} and/or Fisher-Hasselblatt
\cite{fisherHasselblatt12}.  It is easy to see from the construction
presented by Kuznetsov \cite{Kuznetsov2009,kuz2012}, that the
diffeomorphism of the disk proposed by Plykin is diffeotopic to the
identity; see also \cite{plykinOnline} for an animated graphics
presentation.

This allows us to define a smooth $C^\infty$ vector field
$Y=(Y_0,Y_1)$ of the solid torus $\DD\times\sS^1$ which points inward
at the boundary and admitting a cross-sectional disk whose Poincar\'e
first return map is the diffeomorphism $f$ described by Plykin, having
an expanding period three periodic orbit $\{p_0,p_1,p_2\}$; see the
right hand side of Figure~\ref{fig:MS}.

The maximal invariant subset $A=\cap_{t\in\RR}Y_t(\DD\times\sS^1)$ of the
flow generated by $Y$ can be written as the union of two connected
components $P\cup O$, where $P$ is the Plykin (expanding) attractor
and $O$ is a periodic source.

\subsubsection{Attaching the Plykin suspension to the Morse-Smale
  example}
\label{sec:attach-plykin-suspen}

We now ``attach'' this vector field in a neighborhood
$V_s(\epsilon)=B(s,\epsilon)\times\sS^1$ of each sink $s\in S_0$ as
follows, for $\epsilon\in(0,1/100)$ small enough. We find a $C^\infty$
partition of unity $(\Psi_s)_{s\in S_0\cup 0}$ subordinated to the
open cover
$\{V_s(\epsilon):s\in S_0\}\cup R_k\times\sS^1\setminus\cup_{s\in
  S_0}V_s(\epsilon/2)$ so that $\supp\Psi_s\subset V_s(\epsilon)$ and
$\supp\Psi_0\subset R_k\times\sS^1\setminus\cup_{s\in S_0}V_s(\epsilon/2)$. Then
we replace $X$ by the $C^\infty$ vector field
\begin{align*}
  G= \Psi_0\cdot X + \sum_{s\in S_0}
  \Psi_s\cdot \left( s+ \frac{\epsilon}2 Y_0 , Y_1  \right)
\end{align*}
which coincides with a rescaled and traslated version of $Y$ in the
interior of each $V_s(\epsilon)$.

Now the limit set in $R_k\times \sS^1$ with respect to the flow
$\phi_t$ induced by $G$ is given by
\begin{align*}
  L=\closure\{\alpha(x)\cup\omega(x):x\in R_k\times\sS^1\}
  =
  \bigcup_{u\in S_1} u\times \sS^1 \cup\bigcup_{s\in S_0} (P_s\cup O_s),
\end{align*}
where $u\times \sS^1$ is a periodic hyperbolic saddle for each $u\in
S_1$ and $P_s, O_s$ are the Plykin attractor and periodic source
within $V_s(\epsilon)$ for each $s\in S_0$; see Figure~\ref{fig:susPlykins}.

\begin{figure}[htb]
\centering
\includegraphics[width=9cm]{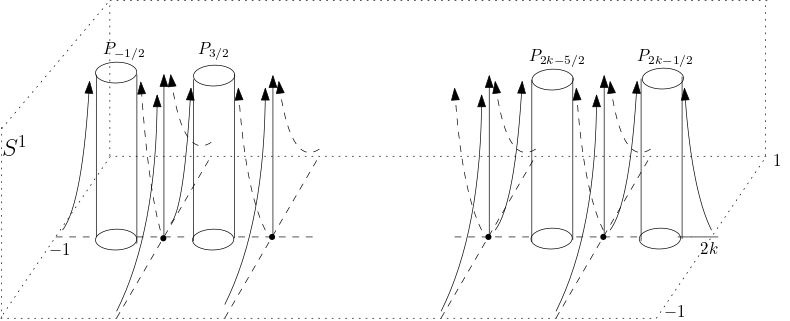}
\caption{\label{fig:susPlykins}A phase portrait of the flow $\phi_t$
  of the vector field $G$, with a suspension of a Plykin attractor in
  the place of each periodic sink of the flow of $X$.}
\end{figure}

If we remove small neighborhoods $U_s$ around each $O_s$, we obtain
that in the connected trapping region
$W=R_k\times \sS^1 \setminus \bigcup_{s\in S_0} U_s$ we have the
attracting set
\begin{align*}
  \Lambda= \bigcap_{t>0}\phi_t(W)
  =\bigcup_{u\in S_1} W^u(u\times\sS^1) \cup\bigcup_{s\in S_0} P_s,
\end{align*}
which is connected by Lemma~\ref{le:connectedLambda}. Moreover,
$\Lambda$ contains $k$ hyperbolic attractors which are suspensions of
the Plykin example.

Since $k>0$ was arbitrarily fixed, we have constructed an attracting
set with no equilibria and any given finite number ($k+1$) of
hyperbolic attractors in a $C^\infty$ smooth $3$-flow.

\subsection{Lorenz-like attracting sets: inequality $s<2 s_L$}
\label{sec:lorenz-attract-set}

From the original work of Lorenz \cite{Lo63} and Sparrow
\cite[Appendix C]{Sp82} it is well known that there exists a
trapping region bounded by an ellipsoid $E$ for an attracting set and
a smaller trapping region bounded by a bitorous $U\subset E$ for the
Lorenz attractor; this was proved later by Tucker using a computer
assisted proof \cite{Tucker,Tu99}.

\begin{figure}[htpb]
\centering
\includegraphics[width=4cm]{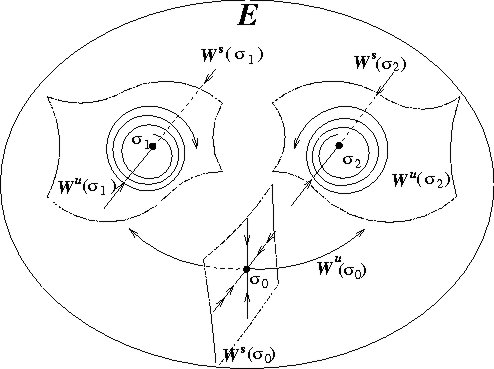}
\quad
\includegraphics[width=4cm]{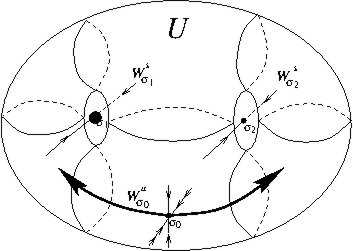}
\caption{\label{fig:bitoro}Local stable and unstable
  manifolds near $\sigma_0,\sigma_1$ and $\sigma_2$, and the
  ellipsoid $E$, on the left; and the trapping bitorus $U$
  on the right.}
\end{figure}

This is depicted in Figure~\ref{fig:bitoro}: the maximal invariant
subset inside $E$ contains, besides the Lorenz attractor, the two
equilibria $\sigma_1,\sigma_2$ with expanding complex eigenvalues
around which the trajectories in the ``lobes'' of the attractor
rotate; see the left hand side of Figure~\ref{fig:Lgeom} for a picture
of the geometric Lorenz attractor \cite{ABS77,Wil79}.

The geometric Lorenz attractor contains a Lorenz-like equilibrium at
the origin and a unique physical measure which is ergodic; see
e.g. \cite[Section 6.3]{Vi97b} or \cite[Section 6]{APPV}. Hence we
obtain an example where the number $s$ of ergodic physical measures
and $s_L$ of Lorenz-like singularities are both equal to one:
$s=1=s_L$.

\begin{figure}[htpb]
\centering
\includegraphics[width=4cm]{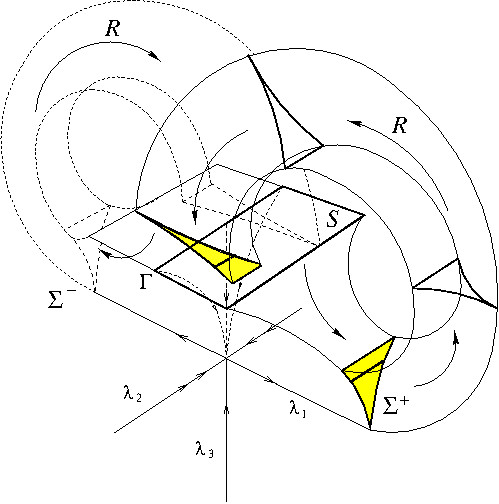}
\quad
\includegraphics[width=6cm]{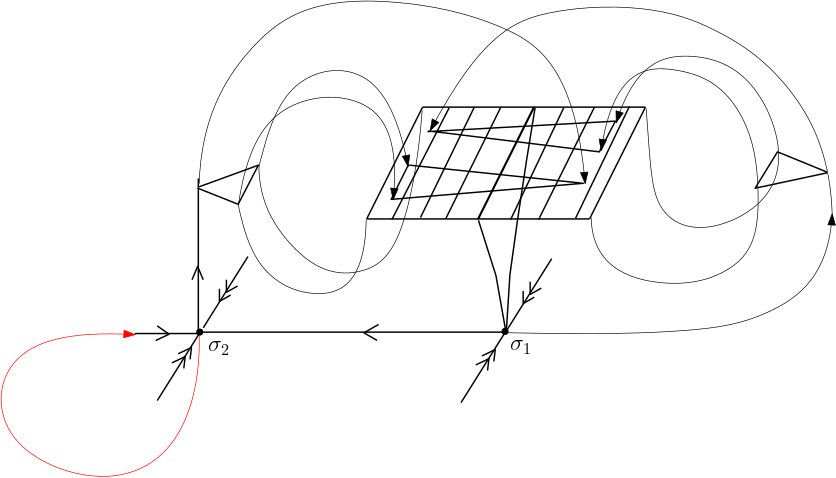}
\caption{\label{fig:Lgeom}The geometric Lorenz attractor on the left,
  with one Lorenz-like equilibrium at the origin and a unique
  (ergodic) physical measure; and an adaptation of the construction on
  the right, providing a singular-hyperbolic attracting set with two
  Lorenz-like equilibria and still a unique physical measure.}
\end{figure}
It is easy to increase the number of Lorenz-like equilibria while
keeping the number of ergodic physical measures, as depicted in the
right hand side of Figure~\ref{fig:Lgeom}, where we have a unique
physical measure and a pair of Lorenz like singularities.

The ``doubling'' of the construction of the geometric Lorenz attractor,
depicted in the left hand side of Figure~\ref{fig:attracting4sing},
provides an attracting set with two transitive geometric Lorenz
components, $H_1$ and $H_2$, above and below the horizontal plane
through the origin. Then the maximal invariant set is
singular-hyperbolic and supports two ergodic physical measures:
$s=2$. Moreover, the set contains three Lorenz-like singularities
marked in the picture: $s_L=3$.

This is a consequence of the existence of a pair of Poincar\'e first
return transformations which have the same properties of the geometric
Lorenz first return map, one for each of the cross-sections presented
in the left hand side of Figure~\ref{fig:attracting4sing}, and so
generate distinct physical measures; see e.g. \cite[Chap. 7, §3]{AraPac2010}.

\begin{figure}[htpb]
\includegraphics[width=6cm]{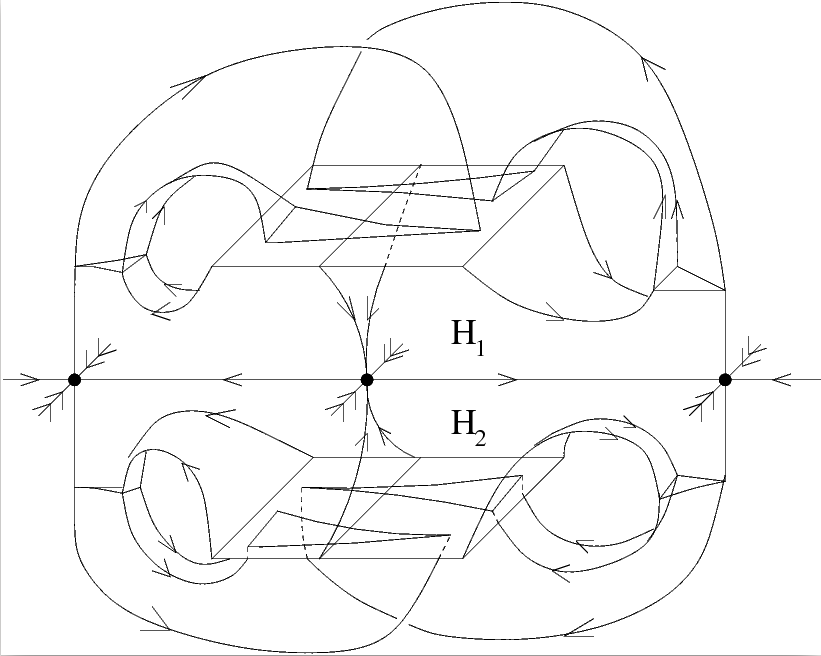}
\includegraphics[width=6cm]{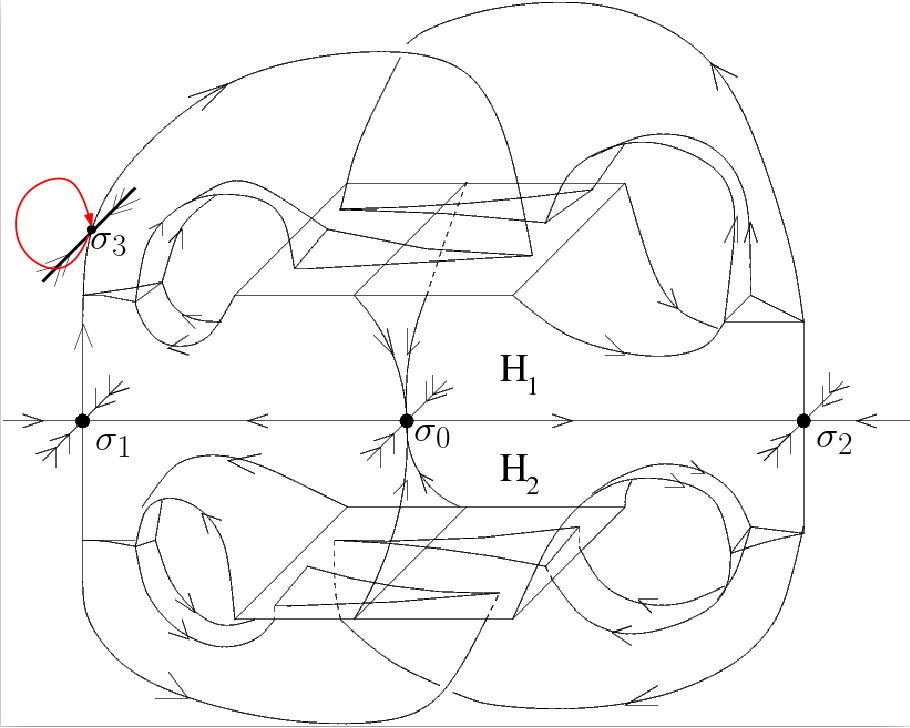}
\caption{\label{fig:attracting4sing} Adaptation of the construction of
  the geometric Lorenz attractor providing a singular-hyperbolic
  attracting set with three Lorenz-like equilibria, having two ergodic
  physical measures, on the left hand sinde. On the right hand side,
  adaptation of this construction with four Lorenz-like and two
  ergodic physical measures.}
\end{figure}

\begin{remark}\label{rmk:moreLorenzlike}
  We can easily increase the number of Lorenz-like equilibria to any
  number we like by a simple adaptation of this construction, as shown
  in the right hand side of Figure~\ref{fig:attracting4sing} where we
  have $s_L=4$ and still only two ergodic physical measures in the
  attracting set. The same trick can be applied to the geometric
  Lorenz construction to obtain any number of Lorenz-like equilibria
  in a singular-hyperbolic attractor, and so with a unique physical
  invariant probability measure.
\end{remark}

\begin{remark}\label{rmk:morenonL}
  We can include in the above construction non-Lorenz-like equilibria
  by extending the trapping region to contain the ``lobes'' around
  which the unstable manifolds of the Lorenz-like equilibria wind.
  This means replacing the solid bitorus by an ellipsoid as the
  trapping region and include the two hyperbolic saddle-focus
  equilibria $\sigma_1, \sigma_2$ in the corresponding attracting set;
  see Figure~\ref{fig:bitoro}.
\end{remark}

\subsection{Example with sharp equality $s=2 s_L=2$}
\label{sec:examples-with-sharp}

Next we again adapt the geometric Lorenz construction to obtain the
equality $s=2\cdot s_L$ with $s_L=1$.

\subsubsection{Overview of the construction}
\label{sec:overvi-constr}

We start with a one-dimensional Lorenz-like transformation with two
expanding fixed repellers at the boundary of the interval; see the
left hand side of Figure~\ref{fig:L1drep}. Then we perform the
geometric Lorenz construction in such a way to obtain this map as the
quotient over the stable leaves of the Poincar\'e first return map to
the global cross-section of a vector field $G_0$; see the right hand
side of Figure~\ref{fig:L1drep}.

  \begin{figure}[htpb]
    \centering
    \includegraphics[width=4cm]{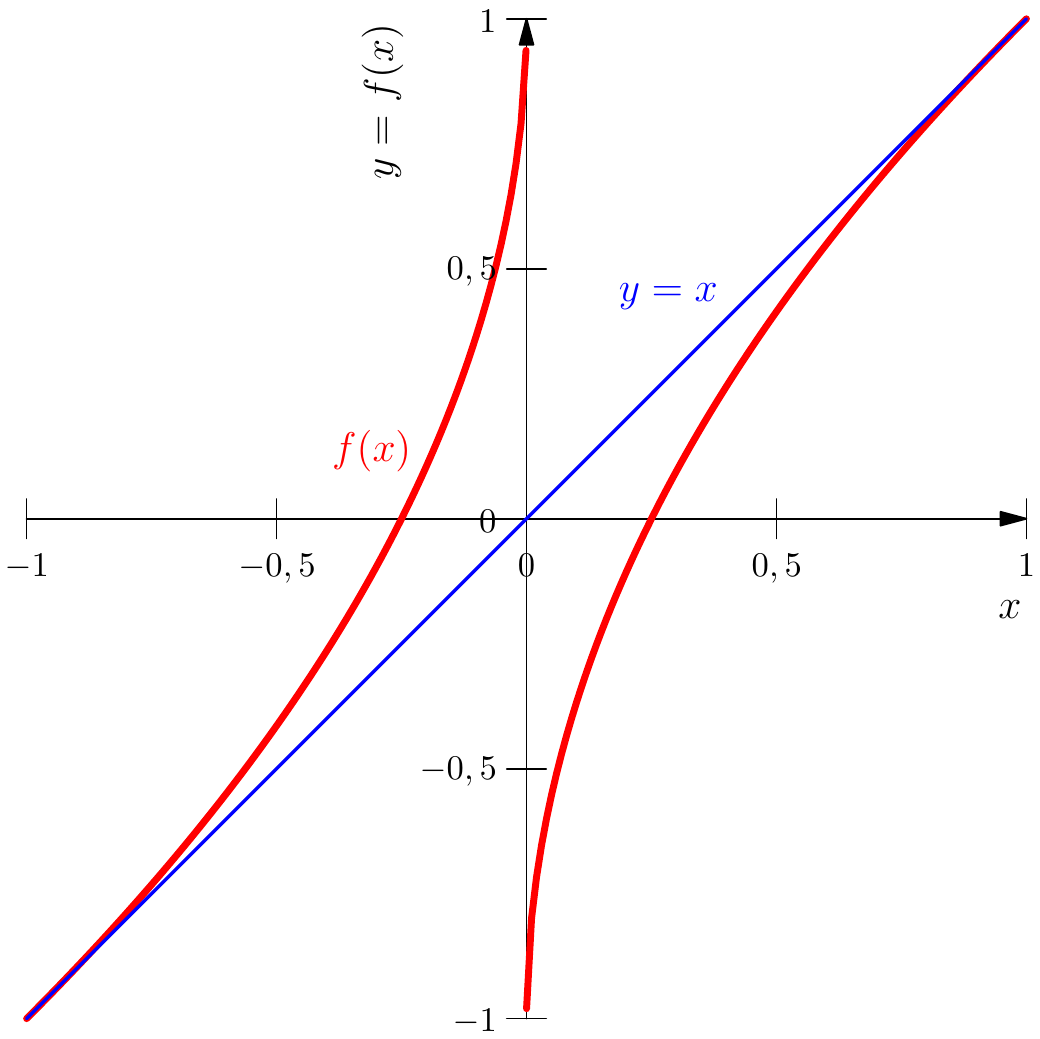}
    \includegraphics[width=5cm]{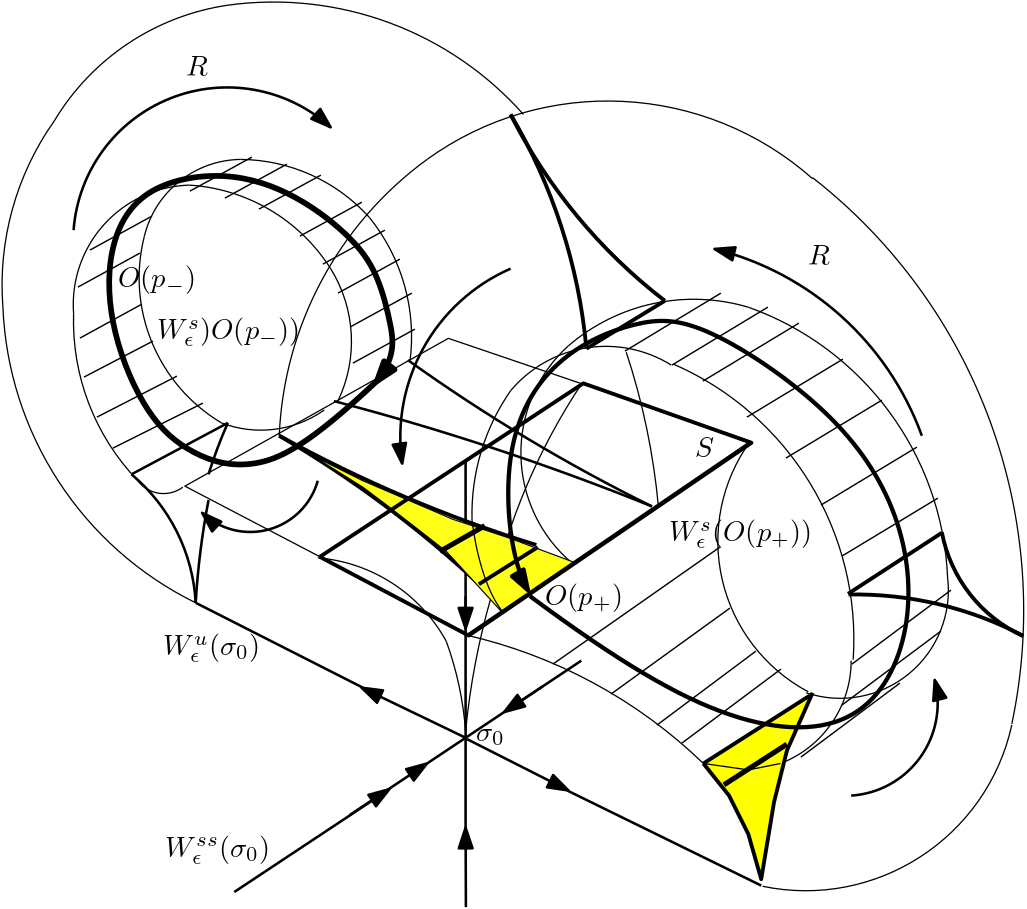}
    \caption{\label{fig:L1drep}Lorenz one-dimensional transformation
      with repelling fixed points at the extremes of the interval on
      the left; and the geometric Lorenz construction with this map as
      the quotient over the contracting invariant foliation on the
      cross-section $S$, with two corresponding periodic saddle-type
      periodic orbits $\cO(p_{\pm})$.}
  \end{figure}

  As usual in the geometric Lorenz construction, we assume that in the
  cube $[-1.1]\times[-1,1]\times[-1,1]$, between the two
  cross-sections $S=[-1,1]^2\times\{1\}$ and
  $S_-=[-1,1]^2\times\{-11\}$, the flow is linear $\dot G_0=A\cdot G_0$
  with $A=\diag\{\lambda_1,\lambda_2,\lambda_3\}$ and a Lorenz-like
  singularity at the origin $\sigma_0$ satisfying
  $\lambda_1<\lambda_3<0<-\lambda_3<\lambda_1$; see e.g. the left hand
  side of Figure~\ref{fig:Lgeom}.
  
  Then we modify the flow to send a neighborhood $D_+$ of the interior
  of a fundamental domain $J$ of the local strong-unstable manifold
  $W^{uu}_{\epsilon}(p_+)$ of a point $p_+$ of the periodic orbit
  $\cO(p_+)$ to the \emph{interior of the cross-section} $S_-$; see
  the left hand side of Figure~\ref{fig:L2d} and the upper half of the
  right hand side of the same figure.

  The boundary of $J$ is a pair of points of a trajectory $\gamma$ of
  the unstable manifold $W^u(p_+)$.  We also send a neighborhood $D_0$
  of a point $q_+$ in the $\omega$-limit of $\gamma$ to
  \begin{enumerate}[(i)]
  \item either the \emph{interior of the cross-section} $S$, so that
    the future trajectories of the points of $D_0$ will accumulate the
    transitive set on $\{z\ge0\}$, as depicted in
    Figure~\ref{fig:L2d};
  \item or to the \emph{interior of another cross-section} of a
    similar distinct attracting set away from $\sigma_0$ (to be able
    to replicate a finite number of similar attracting sets forming a
    connected subset).
  \end{enumerate}
  
  Note that in this way the cusp section $\Sigma_-$, coming from one
  half of $S_-$ in Figure~\ref{fig:L2d}, is eventually sent by the
  flow back to the \emph{interior} of $S_-$.

  To complete the construction of the attracting set:
  \begin{itemize}
  \item we modify the flow to send an open neighborhood of a
    fundamental domain of the local strong-unstable manifold
    $W^{uu}_\epsilon(p_-)$ to the \emph{interior} of $S_-$ in a
    symmetrical way to what as done around $\cO(p_+)$ with respect to
    the origin, by the transformation $T: (x,y,z)\mapsto (-x,-y,z)$
    reminiscent of the symmetry of the original Lorenz system of
    differential equations;
  \item analogously we connect a neighborhood of a point of a
    trajectory inside $W^u_\epsilon(\cO(p_-))$ to
    \begin{enumerate}[(i)]
    \item either the \emph{interior} of the cross-section $S$;
    \item or to the \emph{interior} of another cross-section of a
      similar attracting set away from $\sigma_0$, see below.
  \end{enumerate}
\end{itemize}

\begin{figure}[htpb]
    \centering
    \includegraphics[width=12cm]{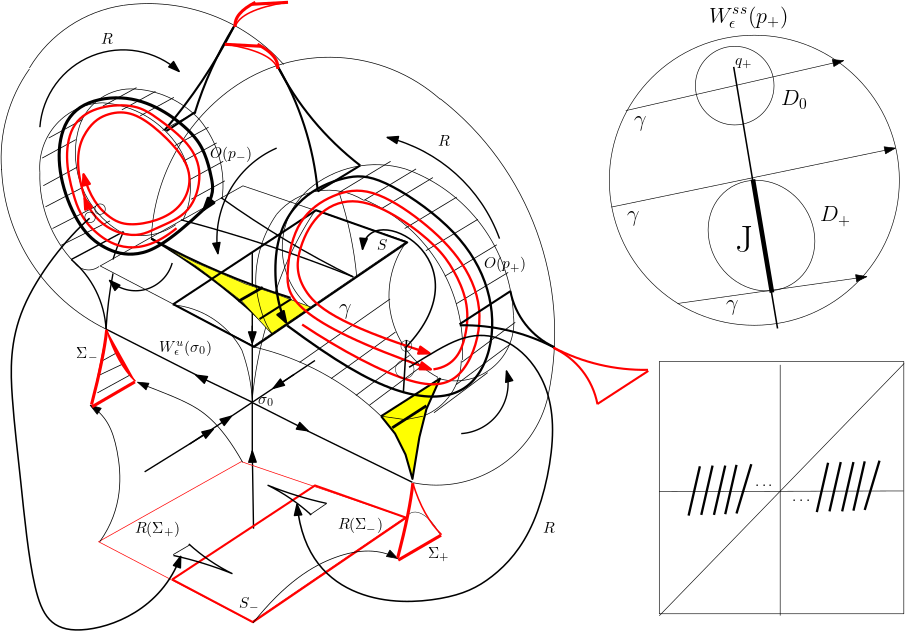}
    \caption{\label{fig:L2d}Connecting the local unstable manifold of
      $\cO(p_+)$ to the interior of the cross-section $S_-$ on the
      left hand side: the connection of the local unstable manifold of
      $\cO(p_-)$ to the interior of $S_-$ is done symmetrically with
      respect to the origin. On the upper right hand side we present a
      detail of the fundamental domain $J$ inside a local
      strong-unstable manifold of a point of $p_+\in\cO(p_+)$; the
      neighborhood $D_+$ of the interior of $J$; and the neighborhood
      $D_0$ of the point $q_+$. On the lower right hand side we
      present a sketch of the return map to $S_-$ after quotienting
      over the contracting leaves.}
  \end{figure}

  We observe that, independently of the choices $(i)$ or $(ii)$, we
  obtain an attracting set accumulating $\sigma_0$ from both ``sides''
  $\{z>0\}$ and $\{z<0\}$, and having two transitive components. The
  return map to $S_-$, after quotienting over the contracting leaves,
  is sketched at the lower right hand side of
  Figure~\ref{fig:L2d}. This attracting set is singular-hyperbolic
  since
  \begin{itemize}
  \item the Poincar\'e first return map to $S$ has the same properties
    as the geometric Lorenz return map; and
  \item the first return map to $S_-$ has a quotient over the
    contracting leaves (whose existence is guaranteed by construction
    similarly to the geometric Lorenz construction) which is piecewise
    expanding with long branches.
  \end{itemize}
  This also provides a pair of ergodic physical measures: $\mu_+$
  whose support is contained in $\{z\ge0\}$; and $\mu_-$ whose support
  intersects $\{z\le0\}$ in a neighborhood of $\sigma_0$. There are
  now the following possibilities to complete the vector field.
  \begin{enumerate}
  \item If we choose $(i)$ above we obtain an attracting set with
    $s=1$ and $s_L=2$; see below for more details.
  \item If we choose $(ii)$ above to connect the unstable manifold of
    either $\cO(p_+)$ or $\cO(p_-)$ (or both) to another similar
    attracting set, then we can obtain examples of connected
    attracting sets with any given number $s_L$ of Lorenz-like
    singularities and precisely $2\cdot s_L$ ergodic physical measures.
  \end{enumerate}
  
\subsubsection{Modifying the geometric Lorenz vector field}
\label{sec:modify-geometr-loren}

More precisely, around the periodic orbits $\cO(p_\pm)$ we extend the
vector field to specify the trapping region as follows; we refer to
Figure~\ref{fig:fundpert}. 

\begin{figure}[htpb]
    \centering
    \includegraphics[width=12cm]{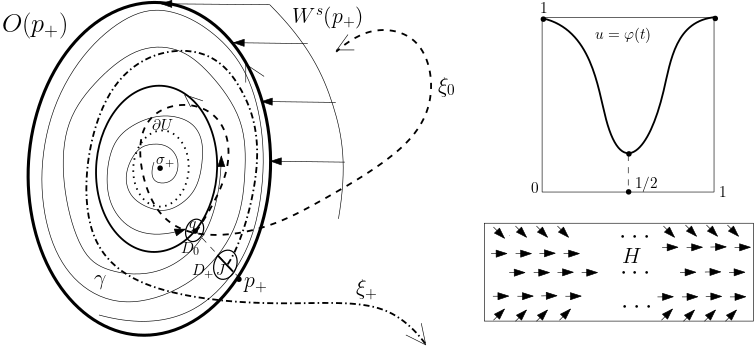}
    \caption{\label{fig:fundpert} The potential defining the vector field to
    be attached to a neighborhood of $\cO(p_+)$ in the upper right hand
    side; a sketch of the perturbation of the vector field in
    $D_0$ and $D_+$ in the left hand side; and a sketch of the vector
    field inside the cylinders in the lower right hand side.}
  \end{figure}

  We consider the $C^\infty$ potential $\vfi$ depicted in the upper
  right hand side of Figure~\ref{fig:fundpert} so that $\vfi<0$ on
  $(0,1)$ and $\vfi(0)=\vfi(1)=0$. We set $Z_0$ to be the vector field
  on the unit disk given by the gradient of $\vfi(x^2+y^2)$ on the
  $xy$-plane.

  Then we attach this vector field to the periodic orbit $\cO(p_+)$ so
  that this orbit corresponds to the repelling periodic orbit of the
  gradient vector field; see the plane vector field at the left hand
  side of Figure~\ref{fig:fundpert}. We note that there exists an
  attracting periodic orbit (periodic sink) for $Z_0$, corresponding
  to the minimum at $1/2$ of $\vfi$.

  We extend the planar field $Z_0$ to a neighborhood in $\RR^3$
  through a contraction on a perpendicular direction, to obtain the
  field $Z_1$ -- we locally identify the stable manifold of $\cO(p_+)$
  with $W^s_{\epsilon_0}(p_+)$ of the initial vector field
  $G_0$. Repeating the construction symmetrically through the
  transformation $T$, this provides a smooth $C^\infty$ vector field
  $G_1$ defined on a ball around the origin.

  At this point, we have a pair of extra saddle-focus equilibria
  $\sigma_+$ and $\sigma_-=T(\sigma_+)$, together with a pair of
  periodic sinks -- incompatible with singular-hyperbolicity -- and
  still no recurrence on $\{z<0\}$ for the vector field $G_1$.  We
  modify $G_1$ by
  \begin{enumerate}
  \item fixing a point $p_+$ in the periodic hyperbolic saddle and
    choosing
    \begin{enumerate}
    \item a fundamental domain $J$ of its local strong-unstable
      manifold $W^{uu}_\epsilon(p_+)$: $J$ is an interval between two
      points of a trajectory $\gamma$ in the local unstable manifold
      $W^u_\epsilon(p_+)$;
    \item an open ball $D_0$ in $\RR^3$ containing the
      interior of $J$ in $W^{uu}_\epsilon(p_+)$ and disjoint from
      $\gamma$;
    \item a point $q_+$ in the periodic sink which is also in the
      closure of $W^{uu}_\epsilon(p_+)$ and in the $\omega$-limit of
      $\gamma$, together with 
    \item a small open ball $D_0$ containing $q_+$ and disjoint from $D_+$.
    \end{enumerate}
    
  \item We consider $C^\infty$ smooth regular curves
    $\xi_+,\xi_0:I=[0,1]\to\RR^3$ as follows:
    \begin{itemize}
    \item $\xi_+$ starting at a point $q_+\in J$ and ending in the
      interior of $S_-$; and 
    \item $\xi_0$ starting at $q$ and ending in the interior of $S$;
    \end{itemize}
    so that they start and end tangent to $G_1$ and move either transversely
    or in the same general direction of $G_1$: for $*\in\{+,0\}$
    \begin{enumerate}
    \item $\dot\xi_*(t)=G_1(\xi_*(t))$ for $t\in\{0,1\}$; and
    \item either $\xi_*\wedge G_1\circ\xi_*\neq0$ or $\langle \xi_* ,
      G_1\circ\xi_*\rangle>0$.
    \end{enumerate}
    
  \item We take open neighborhoods $U_*\supset V_*$ of $\xi_*(I)$ so
    that $U_*\cap D_*=D_*$ and $U_*\supset \closure{V_*}$, for
    $*\in\{+,0\}$. We also assume that $U_*$ contain tubular
    neighborhoods of $\xi_*(I)$ and that we can parameterize $U_*$ as
    a cylinder $\DD\times (0,1)\subset\RR^3$ (where $\DD$ is the open
    unit disk in $\RR^2$). We introduce the ``inward pointing tubular
    vector field'' $H_*$ in $U_*$ for $*\in\{+,0\}$, given by the
    following expression in the coordinates $\DD\times(0,1)$:
    \begin{align*}
      H_*=\left(
      \vec\nabla \vfi\big(40(x^2+y^2)-1/2\big),
      1
      \right),
    \end{align*}
    whose phase portrait is depicted in the lower right hand side of
    Figure~\ref{fig:fundpert}.
    
  \item We let $\psi,\psi_0,\psi_+$ be the partition of unity
    subordinated to the open cover
    $$\RR^3\setminus(U_0\cup U_+), U_0, U_+$$ respectively, so that
    $\supp\psi_*=\closure{U_*}$ and $\psi_*\mid V_*\equiv1$, for
    $*\in\{+,0\}$. We define the perturbation
    \begin{align}\label{eq:defG}
    G=\psi_0\cdot H_0 + \psi_+\cdot H_++\psi\cdot G_1.
    \end{align}
    We note that points in $J\cap D_+$ and in a neighborhood of $q$
    are eventually sent to the other end of the respective cylinder by
    the action of the modified flow, so that, as depicted in
     the left hand side of Figure~\ref{fig:L2d}:
    \begin{enumerate}
    \item the points of $D_+\cap J$ are taken to $S_-$;
   \item points in a neighborhood of $q$ are sent as in $(i)$ to the
     cross-section $S$ as described in the overview.
    \end{enumerate}
  \item Finally, we perform the same construction symmetrically using
    the transformation $T$ around $\cO(p_-)$ to obtain a flow on a
    ball around the origin with the phase portrait depicted in the
    left hand side of Figure~\ref{fig:L2d}.
  \end{enumerate}

  We note that the choice of the curves $\xi_+$ and $\xi_0$ ensures
  that the definition of $G$ through \eqref{eq:defG} avoids the
  introduction of any new equilibrium point, since $G$ is everywhere a
  \emph{linear convex combination of two vector fields which either
    make an acute angle or are transverse}.
  
  The modified vector field $G$ has no periodic sink and has a
  Lorenz-like attracting set containing $\sigma_0$, which accumulates
  this equilibrium only on a neighborhood of $\sigma_0$ in $\{z\ge0\}$,
  and whose Poincar\'e return map to the cross-section $S_-$, after
  quotienting over the contracting stable leaves, is sketched in the
  lower right hand side of Figure~\ref{fig:L2d}.

  The infinitely many branches of this map are due to the fact that
  the cuspidal sections $\Sigma_\pm$ are sent by the flow close to the
  stable manifold of the periodic saddles $\cO(p_\pm)$, which winds
  around the orbit. The Inclination Lemma ensures that \emph{each
    fundamental domain of the local unstable manifold} of $\cO(p_+)$
  in a neighborhood of the local unstable manifold
  $W^u_\epsilon(\sigma_0)$ inside $\Sigma_\pm$ eventually accumulates
  in $D_+$ (and its symmetric $T(D_+)$ corresponding to
  $\cO(p_-)$). These pieces of $\Sigma_\pm$ are then sent back to the
  \emph{interior} of $S_-$ as in the left hand side of
  Figure~\ref{fig:L2d}.

  To complete the example, we specify the trapping region: take a
  large ellipsoid like $E$ in Figure~\ref{fig:bitoro} centered around
  $\sigma_0$ encompassing the recurrence region with negative $z$; and
  remove a neighborhood of the stable manifold of the saddle-foci
  $\sigma_\pm$ obtaining a bitorus like $U$ of
  Figure~\ref{fig:bitoro}, which is a trapping region.

  The maximal invariant subset $\Lambda_G(U)$ is a connected
  singular-hyperbolic attracting set with one Lorenz-like singularity
  $\sigma_0$ and, due to the existence of two independent Poincar\'e
  return maps, we obtain a pair $\mu_\pm$ of ergodic physical measures
  for the flow. The singular-hyperbolicity of the trajectories through
  the cylinders around $\xi_+,\xi_0$ is a consequence of the following
  properties:
  \begin{itemize}
  \item the $\alpha$-limit is one periodic hyperbolic saddle, or the
    singular-hyperbolic transitive set with cross-section $S_-$, or is
    outside of the trapping region; and 
  \item the $\omega$-limit is the singular-hyperbolic transitive set
    with cross-section $S$; and moreover
  \item in dimension three, the dimensions of the singular-hyperbolic
    splitting, at the $\alpha$-limit and at the $\omega$-limit, are
    compatible.
  \end{itemize}
  
  For future use, we can assume without loss of generality that $U$ is
  contained in the cube $Q_0=[-4,4]^3$ and $G$ is defined in $Q_0$.

  \begin{remark}\label{rmk:4torus}
    The boundary of the trapping region of this attracting set is a
    bitorus with two handles around the arcs $\xi_-$ and $T(\xi_-)$,
    thus a $4$-torus.
  \end{remark}

  \subsection{Examples with sharp equality $s=2s_L$ with any given
    $s_L>1$}
\label{sec:examples-with-sharp-1}

To build an attracting set with any given finite number $s_L>1$ of
Lorenz-like singularities and exactly $2s_L$ ergodic physical
measures, we ``copy and paste'' the previous construction $s_L$ times
translated along the $x$-axis and connect the attracting sets around
each equilibria through the perturbation $(ii)$ loosely described in
the previous Subsection~\ref{sec:overvi-constr}; see
Figure~\ref{fig:atracti}. For that we choose the regular curve $\xi_0$
as in the previous section, except that $\xi_0$ now ends at a
cross-section above a different Lorenz-like singularity.

\begin{figure}[htpb]
    \centering
    \includegraphics[width=12cm]{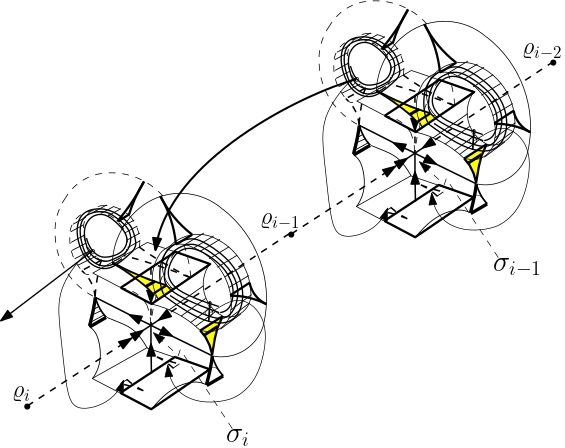}
    \caption{\label{fig:atracti} The equilibria $\sigma_i$ and
      $\sigma_{i-1}$ between repelling equilibria
        $\varrho_i, \varrho_{i-1}$ and $\varrho_{i-2}$ together with
        the nearby maximal invariant subsets; and the connection
        between $D_0^-$ near $\cO(p_{i-1}^-)$ with the cross-section
        near $\sigma_i$.}
  \end{figure}

  More precisely, similarly to the construction presented in
  Subsection~\ref{sec:finiteHypatt}, we start with a Morse-Smale
  vector field $X_0$ defined by the ODE
  $\dot x = g'(x/20) , \dot y =y, \dot z=-z$ on the box
  $R_k=[-10,40k]\times [-5,5]\times[-5,5]$ for some fixed integer
  $k=s_L-1>0$, where $g:\RR\to\RR$ is the same function defined in
  Subsection~\ref{sec:lorenz-attract-set}.

  We obtain $k+1=s_L$ Lorenz-like singularities at
  $S_0=\{\sigma_i=(40i-5,0,0):i=0,\dots,k\}$ and $k$ non-Lorenz-like
  singularities at $S_1=\{\varrho_i=(40i+5,0,0) : i=0,\dots,k-1\}$; see
  Figure~\ref{fig:SpaceMorseSmale}.

\begin{figure}[htpb]
    \centering
    \includegraphics[width=9cm]{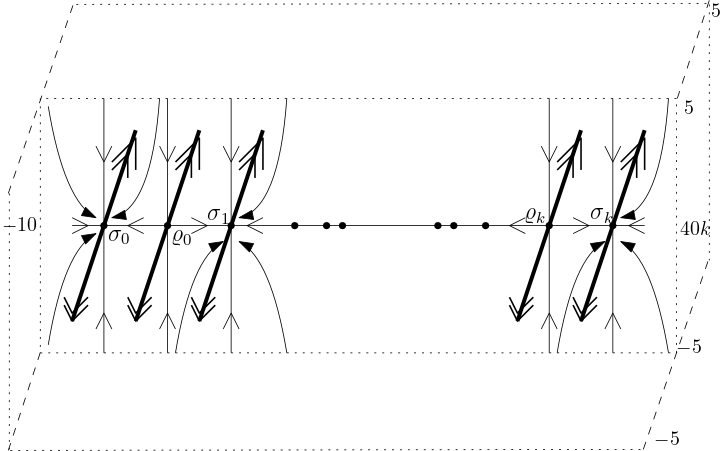}
    \caption{\label{fig:SpaceMorseSmale} The phase portrait of the
      Morse-Smale $3$-dimensional vector field at the base of the
      construction of the example, with Lorenz-like
      $\sigma_i, i=0,\dots,k$ and non-Lorenz-like
      $\varrho_i, i=0,\dots,k-1$ singularities.}
  \end{figure}

  We attach the vector field constructed in
  Subsection~\ref{sec:examples-with-sharp} at every box
  $Q_i=[\sigma_i-4,\sigma_i+4]\times[-4,4]^2$ to the Morse-Smale
  vector field $X_0$, as follows. We consider the open cover
  $W_i= (\sigma_i-5,\sigma_i+5)\times(-4,4)^2$ with $i=0,\dots,k$,
  and $R_k\setminus\sum_i W_i$; and let the corresponding
  $C^\infty$ partition of unity be $\psi_i, i=0,\dots,k$ and $\psi_-$,
  subordinated to this open cover. We define the $C^\infty$ vector
  field
  \begin{align*}
    X_1(w)= \psi_-(w)\cdot X_0(w) +
    \sum_{i=0}^k \psi_i(w)\cdot G(w+\sigma_i) 
  \end{align*}
  on $R_k$. This vector field has $k+1=s_L$ copies of the attracting
  set constructed in Subsection~\ref{sec:examples-with-sharp},
  together with non-Lorenz-like singularities between these attracting
  sets.

  Next we modify $X_1$ into a field $X_2$ so that we obtain a
  \emph{connected} attracting set containing the union of the
  transitive pieces of $X_1$.
  
  The attracting set around the ``last'' Lorenz-like equilibrium
  $\sigma_{s_L}$ will keep the connection from $D_0$ to the
  cross-section ``above'' the same equilibrium. The attracting sets
  around the singularity $\sigma_i$ will be modified according to
  option $(ii)$ (from the overview Subsection~\ref{sec:overvi-constr})
  to have a connection from $D_0$ to the cross-section ``above'' the
  ``next'' equilibrium $\sigma_{i+1}$, for $i=0,\dots,s_L-1$; see
  Figure~\ref{fig:atracti}.
  
\begin{figure}[htpb]
    \centering
    \includegraphics[width=9cm]{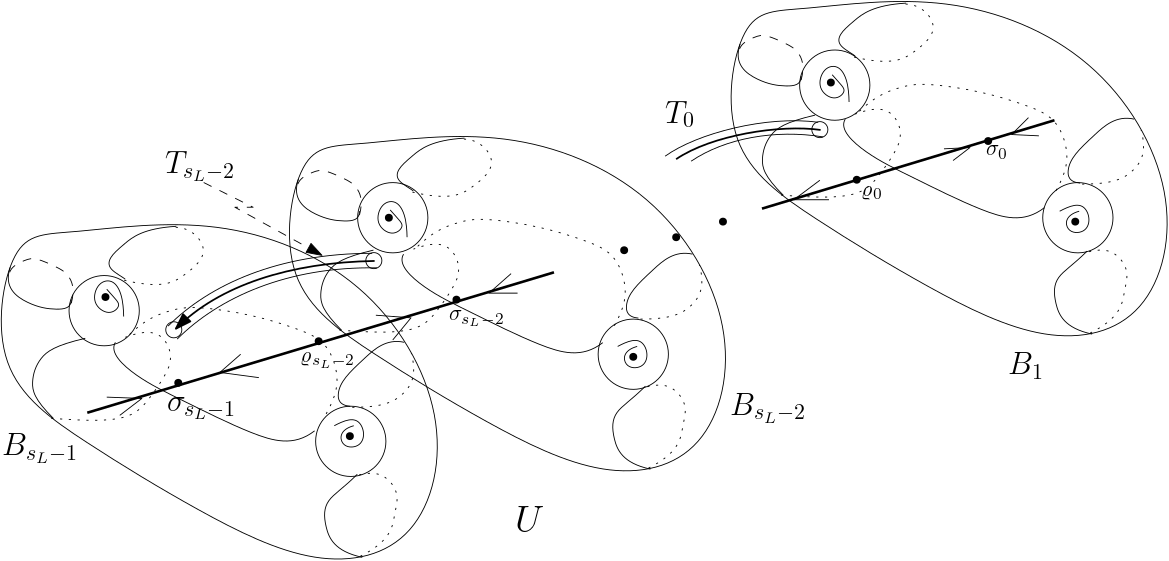}
    \caption{\label{fig:trapping} The trapping region $U$ for the
      attracting set with $2\cdot s_L=s>2$ ergodic physical
      measures, given by  the union of the $4$-tori
      $B_i$ and tubular neighborhoods $T_i$. The repelling equilibria
      are not contained in $U$.}
  \end{figure}

  The trapping neighborhood $U$ will be the union of bitori $B_i$
  shown in Figure~\ref{fig:trapping} around each $\sigma_i$, where a
  neighborhood around all the saddle-foci has been removed, together
  with the ``tubular'' neighborhoods $T_i$ connecting a trajectory of the
  unstable set of $\cO(p_i^-)$ to the cross-section above  $\sigma_{i+1}$
  for $i=0,\dots, s_L-1$.

  This provides exactly a pair of ergodic physical measures containing
  each Lorenz-like equilibrium in their support. The connections
  between the unstable manifolds of periodic orbits in the support of
  one ergodic measure and the cross-section of the flow around a
  different equilibrium, ensure that the maximal invariant subset
  $\Lambda$ within $U$ will be connected.

This completes the proof of Theorem~\ref{mthm:sharpex}.

\begin{remark}\label{rmk:boundary}
  We note that the boundary of the topological basin of attraction of
  this attracting set is diffeomorphic to a $4\cdot s_L$-tori (a
  topologically connected sum of $4s_L$ tori; or a ``torus with $4s_L$
  holes'').
\end{remark}

  \begin{remark}\label{rmk:nonLlike}
    If we do not remove the neighborhoods around the saddle-foci in
    Figure~\ref{fig:trapping}, that is, if we replace each $4$-torus by
    an ellipsoid containing the maximal invariant set around
    $\sigma_i$ and the nearby cross-sections, then we include in the
    new attracting set $\widetilde\Lambda$ all the saddle-focus
    equilibria plus some trajectories connecting these singularities
    with the transitive set accumulating on $\sigma_i$ in $\{z<0\}$.

    This $\widetilde\Lambda$ is a singular-hyperbolic attracting set
    with the same number of Lorenz-like singularities and ergodic
    physical measures as before, but containing a pair of
    non-Lorenz-like hyperbolic equilibria for each Lorenz-like
    equilibrium.
  \end{remark}




  \section{The bound on the number of ergodic physical measures}
\label{sec:singhypexist}

Here we start the proof of Theorem~\ref{mthm:numberSRB} and
Corollary~\ref{mcor:numberSRB}. After some preliminaries, the proof is
presented in Subsection~\ref{sec:number-singul-physic}.

The definition of sectional-hyperbolicity ensures that every invariant
probability measure supported in a sectional-hyperbolic set is a
hyperbolic measure. Moreover, if the vector field is smooth (at least
of class $C^2$) from the proof \cite[Theorem A]{araujo_2021} -- or of
\cite[Theorem B, §4]{APPV} or explicitly from \cite[Theorem
1.5]{Sataev2010} for $d_{cu}=2$ -- we get the first part of the
following statement.

\begin{theorem}
  \label{thm:APPV}
  Every sectional-hyperbolic attracting set for a $C^2$ smooth flow
  admits finitely many $\mu_0, \dots,\mu_k$ ergodic physical invariant
  measures which are SRB measures for the system. Moreover, the union
  of the ergodic basins of these measures covers a full Lebesgue
  measure subset of the topological basin of attraction of $\Lambda$.

  In addition, the support of an ergodic physical/SRB measure without
  singularities contained in $\Lambda$ is a hyperbolic attractor.
\end{theorem}

In fact, the same result is true in higher dimensions; see
\cite{araujo_2021}. The last statement of Theorem~\ref{thm:APPV} is
not contained in \cite{APPV} and for completeness we provide a proof
in the following Subsection~\ref{sec:suppreghyp} together with the
definition of SRB measure.

\subsection{The support of a non-singular physical measure is a
  hyperbolic attractor}
\label{sec:suppreghyp}

Similarly to the uniformly hyperbolic setting (see e.g. \cite[Theorems
7.4.8 \& 7.4.10]{fisherHasselblatt12}), in the sectional-hyperbolic
setting we have that physical and SRB measures coincide.

\begin{theorem}
  \label{thm:appv-attracting}
  Let $\Lambda$ be a sectional-hyperbolic attracting set for
  a $C^2$ vector field $X$ with the open subset $U$ as
  trapping region. Then
  \begin{enumerate}
  \item there are finitely many ergodic physical
    measures $\mu_1,\dots,\mu_k$ supported in $\Lambda$ such
    that the union of their ergodic basins covers $U$
    Lebesgue almost everywhere:
    \begin{align}\label{eq:appv-attracting}
      \m\left(U\setminus\big(\cup_{i=1}^k
      B(\mu_i)\big)\right)=0.
    \end{align}
  \item Moreover, for each $X$-invariant ergodic probability
    measure $\mu$ supported in $\Lambda$ the following are
    equivalent
    \begin{enumerate}
    \item
      $h_\mu(X_1)=\int\log|\det DX_1\mid_{E^{cu}}|\,d\mu>0$;
    \item $\mu$ is a $SRB$ measure, that is, there is a positive
      Lyapunov exponent at $\mu$-a.e. point and $\mu$ admits an
      absolutely continuous disintegration along the corresponding
      unstable manifolds;
    \item $\mu$ is a physical measure, i.e., its basin
      $B(\mu)$ has positive Lebesgue measure.
    \end{enumerate}
  \item In addition, the family $\EE$ of all $X$-invariant
    probability measures which satisfy item (2) above is the
    convex hull
    $ \EE=\{\sum_{i=1}^k t_i \mu_i : \sum_i t_i=1; 0\le
    t_i\le1, i=1,\dots,k\}.  $
  \end{enumerate}
\end{theorem}

\begin{proof}
  This is \cite[Theorem 1.7]{ArSzTr} proved for a sectional-hyperbolic
  attracting set whose center-unstable dimension is two:
  $\dim E^{cu}=2$. The proof is the same in any dimension; see
  \cite{araujo_2021}.
\end{proof}

The following is a consequence of the Hyperbolic
Lemma~\ref{le:hyplemma} together with the absolutely continuous
disintegration of physical/SRB measures.

\begin{proposition}\label{pr:regSRBhypattract}
  Let $\mu$ be a physical/SRB ergodic probability measure whose
  support $\supp\mu$ is contained in a sectional-hyperbolic attracting
  set $\Lambda=\Lambda_G(U)$. If $\supp\mu$ does not contain
  Lorenz-like singularities, then $\supp\mu$ is a hyperbolic
  attractor.
\end{proposition}

\begin{proof}
  Let $\mu$ be a physical ergodic probability measure supported in
  $\Lambda$, and let us assume that $A=\supp\mu$ contains no
  Lorenz-like singularities. Since $A$ is transitive, then by
  Remark~\ref{rmk:notLorenzlike} there can be no other (hyperbolic)
  singularities in $A$. Then the compact invariant subset $A$ is
  uniformly hyperbolic, by the Hyperbolic Lemma~\ref{le:hyplemma}.

  The SRB property can be geometrically described as follows; see e.g
  \cite{PS82}.  For $\mu$-a.e. $x$ there exists a neighborhood $V_x$
  where $\mu$ admits a disintegration
  $\{\mu_\gamma\}_{\gamma\in\cF(V_x)}$ over the strong-unstable leaves
  $\cF(V_x)$ that cross $V_x$ which is absolutely continuous with
  respect to the induced volume measure on the leafs. More precisely,
  we have
  \begin{itemize}
  \item $\mu(\vfi)=\int\mu_{\gamma}(\vfi)\,d\hat\mu(\gamma)$ for each
    bounded measurable observable $\vfi:M\to\RR$, where $\hat\mu$
    is the quotient measure on the leaf space induced by $\mu$ and, in
    addition
  \item $\mu_{\gamma} = \psi_\gamma \Leb_{\gamma}$ for
    $\hat\mu$-a.e. $\gamma$, where $\Leb_{\gamma}$ denotes the measure
    induced on the leaf $\gamma\in\cF(V_x)$ (which is a submanifold of
    $M$) by the Lebesgue volume measure; and
    $\psi_\gamma:\gamma\to[0,+\infty)$ is a strictly positive
    measurable and $\Leb_\gamma$-integrable density function.
  \end{itemize}
  In particular, $\Leb_\gamma$-a.e. point of $\hat\mu$-a.e. leaf in
  $\cF(V_x)$ belongs to $A$. Indeed, given any full measure subset $A_1$
  of $A$, we have that $\mu_\gamma(A_1)=1$ for $\hat\mu$-a.e. $\gamma$
  and hence $\Leb_\gamma(A_1)=1$ also. Since a full Lebesgue measure
  subset is dense and $A$ is closed, we see that the unstable
  leaf $\gamma$ is contained in $A$ for a $\mu$-positive
  measure subset of $V_x$.

  In addition, unstable leaves of inner radius $\epsilon>0$ are
  defined on all points of $A$ by uniform hyperbolicity and the map
  $A\ni x\mapsto W^{uu}_\epsilon(x)$ is continuous in the $C^1$
  topology of disk embeddings; see e.g. \cite[Chapter
  6]{fisherHasselblatt12}.

  Then, since the previous property holds for a full $\mu$-measure subset
  of points $x$, which is dense in $A$, we see that $A$ contains the
  unstable leaves through a dense subset of its points. By continuity
  of the unstable foliation in a hyperbolic set, we conclude that $A$
  contains the unstable manifold through each of its points. This
  ensures that $A$ is a hyperbolic attractor:
  \begin{itemize}
  \item the support of the ergodic measure $\mu$ is topologically
    transitive;
  \item the union the center stable leaves $W^{s}(y)$, through each
    point $y$ of the local strong-unstable leaf $W^{uu}_\epsilon(x)$,
    is a open subset contained in the topological basin of $A$;
  \item the flow is at least $C^2$, hence we can apply
    \cite[Proposition 5.4 \& Theorem 5.6]{BR75}.
  \end{itemize}
  The proof is complete.
\end{proof}


\subsection{Non-existence of hyperbolic attracting sets near
  singular-hyperbolic attractors}
\label{sec:coroll-non-existence}

Here we state and prove a small extension of the statement of
Theorem~\ref{thm:morales}, showing that not only \emph{attractors} but
also \emph{attracting sets} must be singular near a
singular-hyperbolic attractor.

\begin{corollary}\label{cor:Morales}
  Let $\Lambda$ be a singular-hyperbolic attractor of a $3$-flow of a
  $C^r$ vector field $X$, $r \ge 1$. Then, there is a neighborhood $U$
  of $\Lambda$ such that \emph{every attracting set} in $U$ of a $C^r$
  vector field $C^r$ close to $X$ is singular.
\end{corollary}

\begin{proof}
  Let $\Lambda=\Lambda_X(U)$ be a singular-hyperbolic attractor as in
  the statement of the corollary (and Theorem~\ref{thm:morales}) and
  $\V$ be a $C^r$ neighborhood of $X$ satisfying the conclusion of
  Theorem~\ref{thm:morales}: for every $G\in\V$ every attractor
  $\Gamma\subset U$ with respect to $G$ contains some (Lorenz-like)
  singularity -- since these are the only equilibria allowed in
  singular-hyperbolic attractors; see e.g. \cite{MPP04,AraPac2010}.

  Arguing by contradiction, let $A=\Lambda_G(W)$ be an attracting set:
  the maximal invariant subset within $W\subset U$; and let us assume
  that $A$ is non-singular.

  Hence, by the Hyperbolic Lemma~\ref{le:hyplemma}, $A$ becomes a
  hyperbolic locally maximal attracting set for the flow generated by
  $G$. Thus, from the standard uniform hyperbolic theory for flows
  \cite[Corollaries 5.3.21 \& 5.3.22]{fisherHasselblatt12} the set $A$
  satisfies shadowing and periodic orbits in $W$ are dense in $A$.
  Therefore, by ``Spectral Decomposition'' \cite[Proposition
  5.3.33]{fisherHasselblatt12}, $A$ is the disjoint union of finitely
  many basic sets: $A=\sum_i A_i$ where each $A_i$ is compact,
  hyperbolic, transitive and locally maximal.

  In particular, we can find pairwise disjoint open neighborhoods such
  that $A_i\subset W_i\subset W$ and
  $A_i=\bigcap_{t\in\RR}\closure{\phi_t^G(W_i)}$. But if
  $x\in\Lambda_G(W_i)$, then $x\in W_i\subset W$ and
  $\phi_{-t}^G(x)\in W_i\subset W$ for all $t\ge0$. Hence $x\in
  \Lambda_G(W)=A$ and $x\in W_i$; thus $x\cap A\cap W_i=A_i$. This
  shows that $\Lambda_G(W_i)\subset A_i$. Since clearly
  $\Lambda_G(W_i)\supset A_i$ by invariance of $A_i$, we see that
  $A_i$ is a hyperbolic attractor.

  However, such attractors cannot exist by
  Theorem~\ref{thm:morales}. This contradiction shows that every
  attracting set within $U$ must contain some (Lorenz-like)
  singularity for all $G\in\V$, as we wanted to prove.
\end{proof}


\subsection{Number of singular ergodic physical measures}
\label{sec:number-singul-physic}

Here we prove Theorem~\ref{mthm:numberSRB} and
Corollary~\ref{mcor:numberSRB}. 

We present some auxiliary results: the existence of an invariant
stable foliation covering the trapping region of any partially
hyperbolic attracting set; the denseness of the stable manifold of
singularities in the support of an ergodic physical measure; and
finally use these results, considering the ergodic physical measures
whose suport contains some equilibrium, to complete the proofs of the
main Theorem~\ref{mthm:numberSRB} and Corollary~\ref{mcor:numberSRB}.

\subsubsection{Existence of stable foliation covering $U$}
\label{sec:existence-stable-fol}

We recall the following useful property of partially hyperbolic
attracting sets.

\begin{proposition}{\cite[Proposition~3.2]{ArMel17}} \label{prop:Es}
  Let $\Lambda$ be a partially hyperbolic attracting set.  The stable
  bundle $E^s$ over $\Lambda$ extends to a continuous uniformly
  contracting $D\phi_t$-invariant bundle $E^s$ over an open neighborhood
  of $\Lambda$.
\end{proposition}

Let $\cD^k$ denote the $k$-dimensional open unit disk and let
$\mathrm{Emb}^r(\cD^k,M)$ denote the set of $C^r$ embeddings
$\gamma:\cD^k\to M$ endowed with the $C^r$ distance. In the present
setting, we have $d_s=\dim E^s=1$ in what follows.

\begin{proposition}{\cite[Theorem~4.2 and Lemma~4.8]{ArMel17}}\label{prop:Ws}
  Let $\Lambda$ be a partially hyperbolic attracting set.  There
  exists a positively invariant neighborhood $U_0$ of $\Lambda$, and
  constants $C>0$, $\lambda\in(0,1)$, such that the following are
  true:
  \begin{enumerate}
  \item For every point $x \in U_0$ there is a $C^r$ embedded
    $d_s$-dimensional disk $W^s_x\subset M$, with $x\in W^s_x$, such
    that
  \begin{enumerate}
  \item $T_xW^s_x=E^s_x$.
  \item $\phi_t(W^s_x)\subset W^s_{\phi_tx}$ for all $t\ge0$.
  \item $d(\phi_tx,\phi_ty)\le C\lambda^t d(x,y)$ for all $y\in W^s_x$,
    $t\ge0$.
  \end{enumerate}
\item The disks $W^s_x$ depend continuously on $x$ in the $C^0$
  topology: there is a continuous map
  $\gamma:U_0\to {\rm Emb}^0(\cD^{d_s},M)$ such that $\gamma(x)(0)=x$
  and $\gamma(x)(\cD^{d_s})=W^s_x$.  Moreover, there exists $L>0$ such
  that the Lipschitz constant of $\gamma(x)$ is bounded above
  $\Lip\gamma(x)\le L$ for all $x\in U_0$.

\item The family of disks $\{W^s_x:x\in U_0\}$ defines a topological
  foliation of $U_0$.
\end{enumerate}
\end{proposition}

We can naturally assume that $U=U_0$ in what follows (using the flow
invariance of the foliation). 


\subsubsection{Density of stable leaves of equilibria contained in the
  support of an ergodic physical measure}
\label{sec:denstablesing}

We state and prove the following useful property of
singular-hyperbolic attracting sets of smooth $3$-flows.

\begin{proposition}\label{pr:densestablesing}
  Let $\Lambda$ be a connected singular-hyperbolic attracting set for
  a $C^2$-smooth $3$-vector field $G$. Let $\mu$ be an ergodic
  physical measure supported on $\Lambda$ and \emph{whose support
    contains a singularity}. Then the stable manifold of the
  singularity transversely intersects the unstable manifold of every
  periodic orbit in the support of $\mu$.
\end{proposition}

\begin{proof}
  Let $\mu$ be an ergodic physical measure supported in $\Lambda$,
  in the setting of the statement of the proposition.

  It is well-known from the Non-Uniform Hyperbolic Theory (Pesin's
  Theory) that the support of a non-atomic
  hyperbolic ergodic probability measure $\mu$ is contained in a
  homoclinic class of a hyperbolic periodic orbit $\cO(p)$; see
  e.g. \cite[Appendix]{KH95} or \cite[Theorem 15.4.3]{BarPes2007}.

  In particular, the set of periodic points $\per(G)\cap\supp\mu$ is
  dense in $\supp\mu$ and all of them are \emph{homoclinically
    related} -- this conclusion also follows from \cite[Theorem
  1.5(8)]{Sataev2010}. This means, more precisely, that given any
  $p,q\in\per(G)\cap\supp\mu$ there exists a transversal intersection
  between the stable and unstable manifolds of $p,q$:
  \begin{align*}
    W^s(p)\pitchfork W^u(q)\neq\emptyset\neq W^u(p)\pitchfork W^s(q).
  \end{align*}

  \begin{lemma}
    \label{le:unstablesing}
    Fix $p_0\in\per(G)\cap\supp\mu$ and let $J=[a,b]$ be an arc on a
    connected component of $W^{uu}_G(p_0)\setminus\{p_0\}$ with
    $a\neq b$.  Then $H=\closure{\cup_{t>0} \phi_t(J)}$ contains a
    singularity of $\Lambda$.
  \end{lemma}

\begin{proof}
  Since $\mu$ is a SRB measure, for $\mu$-a.e. $x$ we have
  $W^u_x\subset\supp\mu$ and
  $W^u_x\pitchfork W^s(\cO(p))\neq\emptyset$.  Thus by the Inclination
  Lemma (see \cite{PM82}) we have
  $W^u(p)\subset\ov{W^u(x)}\subset\supp\mu$.

  Because every periodic point $p_0\in\supp\mu$ is homoclinically
  related to $p$, then we also have
  $W^{uu}(p_0)\subset W^u(p_0)\subset\supp\mu$.

  Note that $H\subset\closure{W^u_0(p_0)}\subset\supp\mu$ and $H$ is a
  compact invariant set by construction, where $W^u_0(p_0)$ is the
  connected component of $W^u(p_0)\setminus\cO(p_0)$ containing
  $J$. In addition, $H$ is clearly connected, since $H$ is also the
  closure of the orbit of the connected set $J$ under a continuous
  flow.

  If $H$ has no singularities, then $H$ is a compact connected
  hyperbolic set of saddle-type. Moreover, $H$ contains the
  strong-unstable manifolds through any of its points, since every
  point in $H$ is accumulated by forward iterates of the arc $J$.

  This means that $H$ is a hyperbolic attracting set and so
  $H=\supp\mu$ by the existence of a dense regular orbit in
  $\supp\mu$.  Consequently, $H$ contains all singularities of
  $\supp\mu$. This contradiction proves that $H$ must contain a
  singularity of $\supp\mu$.
\end{proof}

Fix $p_0$ and $\sigma\in\sing(G)\cap H\cap\supp\mu$ as in the
statement of Lemma~\ref{le:unstablesing}.  We have shown that
$\closure{W^u(p_0)}\cap\sing(X)\neq\emptyset$. Moreover, the flow of $G$ is
inwardly transverse to the boundary of the manifold $\ov{U}$ and has
$\Lambda$ as its maximal invariant subset, which is
singular-hyperbolic. Hence, $G$ in $\ov{U}$ is a sectional-Anosov flow
and we are in the setting of the next result from Bautista and Morales
\cite{BM2010}.

\begin{theorem}{\cite[Corollary 1.4]{BM2010}}
  \label{thm:cor1.4BM}
  If $\cO$ is a periodic orbit of a sectional-Anosov flow $G$ on a
  compact manifold satisfying
  $\closure{W^u(p_0)}\cap \sing(G)\neq\emptyset$, then there exists
  $\sigma\in\sing(G)$ such that
  $W^u(\cO)\pitchfork W^s(\sigma)\neq\emptyset$.
\end{theorem}

Now we take $\cO=\cO_G(p_0)$ and since $G$ is a sectional-Anosov flow
in the trapping region of $\Lambda$, we apply
Theorem~\ref{thm:cor1.4BM} to obtain the existence of
$\sigma\in \sing(G)$ satisfying
$W^u(\cO)\cap W^s(\sigma)\neq\emptyset$. By definition of
singular-hyperbolicity, this implies that
$W^u(p_0)\pitchfork W^s(\sigma)\neq\emptyset$.

This is enough to conclude the proof of Proposition
\ref{pr:densestablesing}. Indeed, since all periodic orbits in
$\supp\mu$ are homoclinically related, it is enough to obtain
$W^u(p_0)\pitchfork W^s(\sigma)\neq\emptyset$ for one periodic point
$p_0\in\supp\mu$.
\end{proof}


\subsubsection{Proof of the main theorem and corollary}
\label{sec:atmost2}

Here we show that a Lorenz-like equilibrium can be accumulated by at
most two supports of ergodic physical measures inside a
singular-hyperbolic attracting set.

\begin{lemma}\label{le:221}
  Given a Lorenz-like singularity $\sigma_0$ in a singular-hyperbolic
  attracting set $\Lambda=\Lambda_G(U)$ in a trapping region $U$ for a
  $3$-dimensional $C^2$ vector field $G$, then there are at most two
  distinct ergodic physical measures whose support contains $\sigma_0$.
\end{lemma}

Using this we are ready to prove Theorem~\ref{mthm:numberSRB}.

\begin{proof}[ Proof of Theorem~\ref{mthm:numberSRB}]
  Let then $\cP$ be the family of ergodic physical measures supported
  in $\Lambda$ whose support contains some (Lorenz-like) singularity,
  and $\cL$ be the set of Lorenz-like singularities in
  $\Lambda$. Proposition~\ref{prop:generaLorenzlike} ensures that
  there is a relation $R$ between $\cP$ and $\cL$ associating to each
  element of $\cP$ the singularities it contains from $\cL$:
  \begin{align*}
    \mu R \sigma \iff \sigma \in\supp\mu.
  \end{align*}
  Lemma~\ref{le:221} implies that
  $R^{-1}\sigma:=\{\eta\in \cP: \eta R \sigma\}$ satisfies
  $n(\sigma):=\#\big( R^{-1}\sigma\big)\in\{0,1,2\}$. Hence we obtain
$
    \cP = \bigcup_{\sigma\in \cL} R^{-1}\sigma
 $
  and since the union is not necessarily disjoint, we conclude
  \begin{align*}
    s=\#\cP \le \sum_{\sigma\in\cL} n(\sigma) \le 2\#\cL=2\cdot s_L.
  \end{align*}
  This completes the proof, depending only on Lemma~\ref{le:221}.
\end{proof}

The following is straightforward.

\begin{proof}[Proof of Corollary~\ref{mcor:numberSRB}]
  In a trapping neighborhood $U$ of a singular-hyperbolic attractor
  $\Lambda_X(U)$ of a $C^1$-smooth $3$-vector field $X$, there exists a
  $C^1$ neighborhood $\U$ of $X$ so that for all $G\in\U$
  there are no hyperbolic attracting sets with respect to $G$, after
  Corollary~\ref{cor:Morales}.

  Hence, all attracting sets for $G\in\U$ are
  singular-hyperbolic and contain the continuation of some equilibrium
  from the original attractor $\Lambda$, which is necessarily
  Lorenz-like, by Proposition~\ref{prop:generaLorenzlike}.

  Therefore, the number of singularities in $\Lambda_G(U)$ is the same
  as the number of Lorenz-like singularities in the set. Hence, for a
  $C^2$ vector field $G\in\U$ (these form a $C^1$ dense subset of
  $\U$; see e.g. \cite[Chapter 0]{PM82}) we can apply
  Theorem~\ref{mthm:numberSRB} to obtain that the number of ergodic
  physical measures supported in $\Lambda=\Lambda_G(U)$ is bounded by
  twice the number of singularities in $\Lambda$.
\end{proof}

To finish, we present a proof of Lemma~\ref{le:221}.

\begin{proof}[Proof of Lemma~\ref{le:221}]
  Given an ergodic physical probability measure $\mu_1$ so that
  $\supp\mu_1\subset\Lambda$ and $\sigma\in\supp\mu_1\cap\sing(G)$,
  then Proposition~\ref{pr:densestablesing} implies that the unstable
  manifold of $\mu_1$-a.e.  point crosses the stable manifold of
  $\sigma$ and so accumulates $\sigma$.

  Indeed, ergodic physical measures are (non-uniformly) hyperbolic
  measures and also non-atomic, because they are also SRB measures.
  Hence their support is contained in a homoclinic class of a
  hyperbolic periodic orbit $\cO(p)$; see \cite[Theorem
  15.4.3]{BarPes2007}.  In our singular-hyperbolic setting these
  periodic orbits are hyperbolic saddles with index $1$\footnote{The
    index of a hyperbolic periodic orbit is the dimension of the
    stable direction}.
  
  This ensures that for $\mu$-a.e. $x$ there exists
  $p\in\per(G)\cap\Lambda$ so that
  $W^s(p)\pitchfork W^u(x)\neq\emptyset$.  Since
  $W^u(p)\pitchfork W^s(\sigma_0)\neq\emptyset$, the Inclination Lemma
  now ensures that we also have
  $W^u(x)\pitchfork W^s(\sigma)\neq\emptyset$.

  Thus we find a center unstable leaf $W_1$ through a $\mu_1$ generic
  point which transversely crosses the local stable manifold of
  $\sigma$ and so, by the Inclination Lemma, become arbitrarily
  close to $\sigma$.

  According to Remark~\ref{rmk:notLorenzlike}, singular-hyperbolic
  attracting sets satisfy $W^{ss}(\sigma)\cap\Lambda=\{\sigma\}$ for
  all Lorenz-like singularities $\sigma\in\sing(G)$. Then the
  accumulation of $\sigma$ by the trajectory of $W_1$ is performed along the
  one-dimensional weak-stable direction of $\sigma_0$; see the left
  hand side of Figure \ref{fig:transversalWs}.

  We can repeat the above arguments for any other ergodic physical
  measure $\mu_2$ obtaining a center unstable leaf $W_2$ through a
  $\mu_2$ generic point which also transversely crosses the stable
  manifold of some Lorenz-like equilibria $\sigma'$.

  If $\sigma=\sigma'$, then the future trajectories of $W_i,i=1,2$
  both accumulate $\sigma$ and so accumulate each other; see Figure
  \ref{fig:transversalWs}.

\begin{figure}[htpb]
  \centering
  \includegraphics[width=10cm]{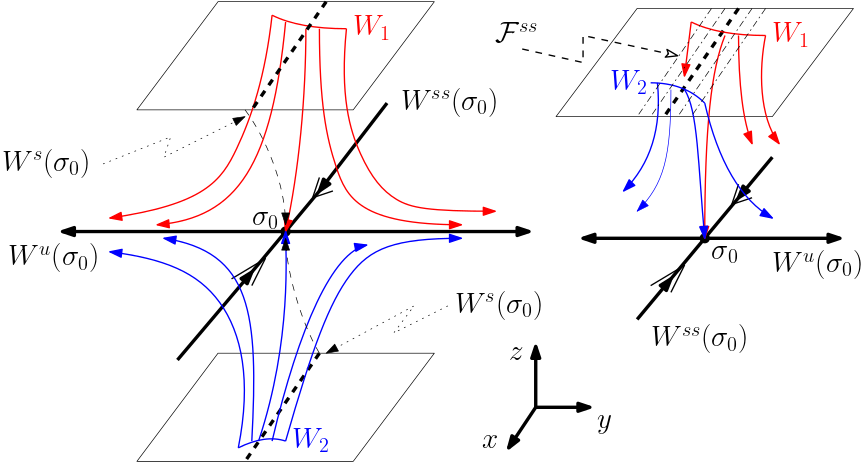}
  \caption{\label{fig:transversalWs} On the left hand side:
    the unstable manifolds $W_1,W_2$ associated to generic
    points of $\mu_1,\mu_2$ on opposite sides of
    $W^s(\sigma_0)$. On the right hand side: the unstable
    manifolds $W_1,W_2$ associated to generic points of
    $\mu_1,\mu_2$ on the same side of $W^s(\sigma_0)$ and a
    cross-section to the flow very close to $\sigma_0$. }
\end{figure}

We note that the local strong-stable manifold $W^{ss}_{loc}(\sigma_0)$
divides the local stable manifold $W^s_{loc}(\sigma_0)$ in two
``sides''.  More precisely, let $h:B(\sigma_0,r)\to T_{\sigma_0}M$ for
some $r>0$ be the homeomorphism conjugating
$h\circ\phi_t=e^{t DG_{\sigma_0}}\cdot h$ the flow $\phi_t$ of $G$
with the linear flow given by $DG_{\sigma_0}$ according to the
Hartman-Grobman Theorem; see e.g. \cite{PM82}. By a linear change of
coordinates on $T_{\sigma_0}M$, we may assume without loss of
generality that the eigenspaces associated to the eigenvalues
$\lambda_1\le\lambda_2<0<\lambda_3$ are respectively the $x$, $z$ and
$y$ axis, as depicted in Figure~\ref{fig:transversalWs}. Then we set,
for a small enough $\rho>0$
\begin{align*}
  B(\sigma_0)^+=h^{-1}\big(B(0,\rho)\cap\{z>0\}\big)
  \qand
  B(\sigma_0)^-=h^{-1}\big(B(0,\rho)\cap\{z<0\}\big)
\end{align*}
the ``top'' and ``down'' sides of a neighborhood of the singularity
$\sigma_0$.

Now we observe that
\emph{if $\sigma_{\text{0}}$ is accumulated by $W_1$ and $W_2$ along
  the same side of the weak-stable direction}, as in the right hand
side of Figure~\ref{fig:transversalWs} then, since
\begin{itemize}
\item both unstable leaves belong to $\Lambda$ (because it
  is an attracting set), and
\item the stable leaves through points of $\Lambda$ have
  uniform size (by Proposition~\ref{prop:Ws}),
\end{itemize}
we conclude that the stable foliation $\cF^{ss}$ of
$\Lambda$ transversely crosses both leaves near $\sigma_0$.


Moreover, it is well-known that the stable foliation for $C^2$
partially hyperbolic flows is absolutely continuous (see
e.g. \cite{PS82} and \cite{ArMel17})), and so a positive Lebesgue
measure subset of $W_1$ will be sent through the stable holonomy to a
positive Lebesgue measure subset of $W_2$.  Since $\mu_i$ is an
ergodic SRB measure, a full measure subset of $W_i$ is formed by
$\mu_{\text{i}}$-generic points, $i=1,2$.  Then we can find a
$\mu_1$-generic point $x_1$ and a $\mu_2$-generic point $x_2$ in the
same stable leaf. This implies that $\mu_1=\mu_2$.

\emph{The only possible way to avoid this phenomenon} is for
$W_1,W_2$ to accumulate $\sigma_0$ along different sides of
the local stable manifold of $\sigma_0$, as in the left hand
side of Figure~\ref{fig:transversalWs}.

However, if there existed still another ergodic physical measure
$\mu_3$ whose support contains the same equilibrium $\sigma_0$, then
some pair of the measures $\mu_i, i=1,2,3$ would accumulate $\sigma_0$
through central-unstable leaves along the same side of $\sigma_0$, and
so this pair of measures would coincide.

This shows that at most a pair of distinct ergodic physical measures
can share a singularity on their respective supports. The proof of the
lemma is complete.
\end{proof}

\begin{figure}[htpb]
  \centering
  \includegraphics[width=8cm]{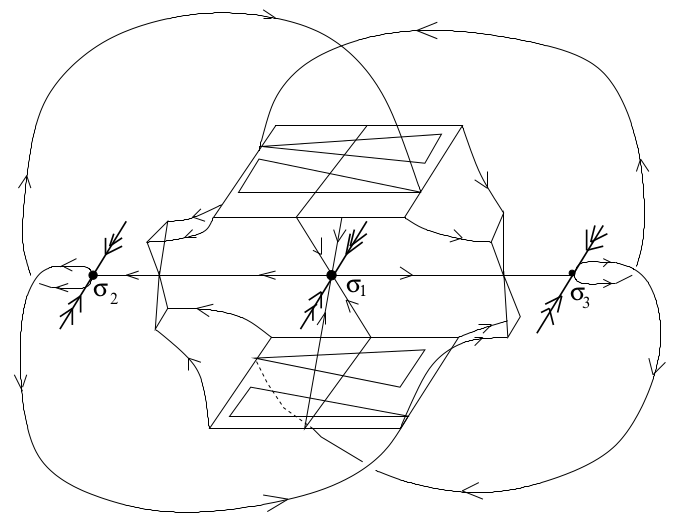}
  \caption{\label{fig:trans2sides}Singular-hyperbolic attractor whose
    singularities are accumulated by both sides by the support of the
    unique physical measure.}
\end{figure}

\begin{remark}\label{rmk:doubleacc}
  There are examples of singular-hyperbolic attractors with
  Lorenz-like singularities accumulated on both sides; see e.g.  the
  proof of~\cite[Theorem B, pg. 345-346]{Morales07} and
  Figure~\ref{fig:trans2sides}.
\end{remark}



\def\cprime{$'$}

\bibliographystyle{abbrv}


\end{document}